\documentclass{amsart}

\title{A $p$-adic interpretation of some integral identities for Hall-Littlewood polynomials}
\author{Vidya Venkateswaran}
\address{Department of Mathematics, Massachusetts Institute of Technology, Cambridge, MA 02139}
\email{\href{mailto:vidyav@caltech.edu}{vidyav@math.mit.edu}}
\thanks{Research supported by NSF Mathematical Sciences Postdoctoral Research Fellowship DMS-1204900}

\usepackage{color}
\usepackage[usenames,dvipsnames,svgnames,table]{xcolor}
\definecolor{Blue}{rgb}{0.3, 0.3, 0.9}

\usepackage{hyperref}
\usepackage[svgnames]{xcolor}
\hypersetup{colorlinks,breaklinks,
            linkcolor=DarkBlue,urlcolor=DarkBlue,
            anchorcolor=DarkBlue,citecolor=DarkBlue}
\usepackage{url}
\usepackage{euscript}
\usepackage[verbose,letterpaper,tmargin=1in,bmargin=1in,lmargin=1in,rmargin=1in]{geometry}
\usepackage{amssymb}
\usepackage{amsmath}

\usepackage{xypic}
\usepackage{verbatim}

\newtheorem{theorem}{Theorem}[section]

\newtheorem{lemma}[theorem]{Lemma}
\newtheorem{proposition}[theorem]{Proposition}

\theoremstyle{definition}
\newtheorem{definition}[theorem]{Definition}
\newtheorem*{example}{Example}

\theoremstyle{remark}
\newtheorem*{remarks}{Remarks}

\begin{document}

\maketitle

\begin{abstract}
If one restricts an irreducible representation $V_{\lambda}$ of $Gl_{2n}$ to the orthogonal group (respectively the symplectic group), the trivial representation appears with multiplicity one if and only if all parts of $\lambda$ are even (resp. the conjugate partition $\lambda'$ is even).  One can rephrase this statement as an integral identity involving Schur functions, the corresponding characters.  Rains and Vazirani considered $q,t$-generalizations of such integral identities, and proved them using affine Hecke algebra techniques.  In a recent paper, we investigated the $q=0$ limit (Hall-Littlewood), and provided direct combinatorial arguments for these identities; this approach led to various generalizations and a finite-dimensional analog of a recent summation identity of Warnaar.  In this paper, we reformulate some of these results using $p$-adic representation theory; this parallels the representation-theoretic interpretation in the Schur case.  The nonzero values of the identities are interpreted as certain $p$-adic measure counts.  This approach provides a $p$-adic interpretation of these identities (and a new identity), as well as independent proofs.  As an application, we obtain a new Littlewood summation identity that generalizes a classical result due to Littlewood and Macdonald.  Finally, our $p$-adic method also leads to a generalized integral identity in terms of Littlewood-Richardson coefficients and Hall polynomials.

\end{abstract}

\section{Introduction}
A crucial problem in representation theory can be described in the following way: let $G$ and $H$ be complex algebraic groups, with an embedding $H \hookrightarrow G$.  Also let $V$ be a completely reducible representation of $G$, and $W$ an irreducible representation of $H$.  What information can one obtain about $[V,W] := \text{dim } \text{Hom}_{H}(W,V)$, the multiplicity of $W$ in $V$? Here $V$ is viewed as a representation of $H$ by restriction.  Such branching rules have important connections to physics as well as other areas of mathematics.  There are often beautiful combinatorial objects describing these multiplicities.  One prototypical example is that of the symmetric groups $G = S_{n}$ and $H= S_{n-1}$: the resulting rule has a particularly nice description in terms of Young tableaux.  

Two particularly interesting examples involving matrix groups are the restriction of $Gl(2n)$ to $Sp(2n)$ (the symplectic group) and $Gl(n)$ to $O(n)$ (the orthogonal group); the combinatorics of these branching rules was first developed by D. Littlewood and continues to be a well-studied and active area at the forefront of algebraic combinatorics and invariant theory.  These pairs are also important because they give examples of symmetric spaces.  That is, $G$ is a reductive algebraic group and $H$ is the fixed point set of an involution on $G$; $S = G/H$ is the resulting symmetric space.  The multiplicities in these branching rules are given in terms of Littlewood-Richardson coefficients, another important entity described in terms of tableaux and and lattice permutations.  In fact, since Schur functions are characters of irreducible polynomial representations of $Gl(2n)$, one may rephrase these rules in terms of Schur functions and symplectic characters (respectively, orthogonal characters).  This gives the following integral identities 
\begin{theorem} \label{symplecticgp}  \cite{H, Mac}
(1) For any even integer $n \geq 0$, we have
\begin{align*}
\int_{S \in Sp(n)} s_{\lambda}(S) dS
=\begin{cases} 1, &\text{if all parts of $\lambda$ have even multiplicity} \\
0, & \text{otherwise}
\end{cases}
\end{align*}
(where the integral is with respect to Haar measure on the symplectic group).

(2) For any integer $n \geq 0$ and partition $\lambda$ with at most $n$ parts, we have 
\begin{align*}
\int_{O \in O(n)} s_{\lambda}(O) dO
= \begin{cases} 1, &\text{if all parts of $\lambda$ are even } \\
0, & \text{otherwise}
\end{cases}
\end{align*}
(where the integral is with respect to Haar measure on the orthogonal group).
\end{theorem}
\noindent Proofs of these identities may be found in \cite{Mac}; they involve structure results for the two Gelfand pairs $(GL_{n}(\mathbb{H}), U_{n}(\mathbb{H}))$ and $(GL_{n}(\mathbb{R}), O_{n}(\mathbb{R}))$.  Using the Weyl integration formula, we may rephrase the above identities in terms of the eigenvalue densities for the orthogonal and symplectic groups.  For example, the left hand side of the symplectic integral above can be rewritten as
\begin{equation} \label{schursymplectic}
\frac{1}{2^{n}n!}\int_{T} s_{\lambda}(z_{1}, z_{1}^{-1}, z_{2}, z_{2}^{-1}, \dots, z_{n}, z_{n}^{-1})
 \prod_{1 \leq i \leq n} |z_{i} - z_{i}^{-1}|^{2} \prod_{1 \leq i<j \leq n} |z_{i} + z_{i}^{-1} - z_{j} - z_{j}^{-1}|^{2} dT,
\end{equation}
where 
\begin{align*}
T &= \{ (z_{1}, \dots, z_{n}) : |z_{1}| = \dots = |z_{n}| = 1 \} \\
dT &= \prod_{j} \frac{dz_{j}}{2 \pi \sqrt{-1} z_{j}}
\end{align*}
are the $n$-torus and Haar measure, respectively.  

Macdonald polynomials, $P_{\lambda}(x;q,t)$, are an important family of symmetric polynomials generalizing the Schur polynomials ($P_{\lambda}(x;t,t) = s_{\lambda}(x)$) (see \cite{Mac}).  In \cite{RV}, Rains and Vazirani provided $(q,t)$-generalizations of the restriction identities for Schur functions.  That is, they exhibited densities such that when one integrates a Macdonald polynomial against it (over the $n$-torus), the integral vanishes unless the indexing partition satisfies an explicit condition.  Moreover, when $q=t$, one obtains a known Schur identity.  In fact, they were also able to find Macdonald identities with interesting vanishing conditions, but whose significance is unknown at the Schur level.  To prove these results, Rains and Vazirani used techniques involving affine Hecke algebras; however, this method does not work directly at $q=0$ (another important special case of Macdonald polynomials: the Hall-Littlewood polynomials, $P_{\lambda}(x;t)$), although one can obtain the results as a limit.  

In previous work \cite{VV}, we provided a combinatorial approach for proving the results of Rains and Vazirani at $q=0$; this method allowed for several generalizations, one of which provided a connection with a summation identity of Warnaar \cite{W}.  We were also able to use this approach to settle some conjectures of Rains at the $q=0$ level.  In some cases, the affine Hecke algebra technique of Rains and Vazirani allowed them to determine when a given integral vanishes, but did not yield the nonzero values in the case that the integral is non-vanishing.  Using our method, we were able to compute these values explicitly.  

This paper provides an interpretation of the results of \cite{VV} in terms of $p$-adic representation theory.  The motivation for this connection stems from the appearance of Hall-Littlewood polynomials in the representation theory of $p$-adic groups \cite{MacP}, \cite[Ch. V]{Mac}.  In particular, let $G = Gl_{n}(\mathbb{Q}_{p})$, and let $K = Gl_{n}(\mathbb{Z}_{p})$ be its maximal compact subgroup.  Then $G/K$ is the affine Grassmannian and the spherical Hecke algebra $\mathcal{H}(G,K)$ is the convolution algebra of compactly supported, $K$-bi-invariant, complex valued functions on $G$; it has a basis given by $\{c_{\lambda}\}_{l(\lambda)\leq n}$, where $c_{\lambda}$ is the characteristic function of the double coset $Kp^{\lambda}K$ and $p^{\lambda} = diag(p^{\lambda_{1}}, \dotsc, p^{\lambda_{n}})$.  Macdonald provides a Plancherel theorem in this context, where the zonal spherical functions are given in terms of Hall-Littlewood polynomials with $t=p^{-1}$.  One consequence of this is another interpretation of the statement of Hall-Littlewood orthogonality:
\begin{proposition}\cite{MacP}, \cite[Ch. V]{Mac} For partitions $\lambda, \mu$ of length at most $n$, we have
\begin{equation*}
\int_{T} P_{\lambda}(z_{1}, \dots, z_{n};p^{-1})P_{\mu}(z_{1}^{-1}, \dots, z_{n}^{-1}; p^{-1}) \tilde \Delta_{S}^{(n)}(z;p^{-1}) dT 
= \frac{n!}{v_{n}(p^{-1})} p^{-\langle \lambda, \rho\rangle - \langle \mu, \rho \rangle}\int_{G} c_{\lambda}(g) c_{\mu}(g) dg,
\end{equation*}
where $\rho = \frac{1}{2}(n-1, n-3, \dots, 1-n)$ and $v_{n}(p^{-1}) = \big(\prod_{i=1}^{n} (1-p^{-i})\big)/(1-p^{-1})^{n}$.
\end{proposition}
\noindent Here $\tilde \Delta_{S}$ is the symmetric $q=0$ Macdonald density \cite{RV}.  Since the double cosets $Kp^{\lambda}K$ and $Kp^{\mu}K$ do not intersect unless $\lambda = \mu$, the right hand side vanishes unless $\lambda = \mu$.  In the case $\lambda = \mu$, one may also compute the $p$-adic measure of $Kp^{\lambda}K$ using \cite{MacP}; in particular, this provides an alternate approach for computing the left-hand side of the integral.  Given the structural similarity between orthogonality and the vanishing results of \cite{VV}, we were lead to search for $p$-adic interpretations of the latter results.   

In this paper, we show that the vanishing results for Hall-Littlewood polynomials have a $p$-adic interpretation analogous to that of the Schur branching rules.  We also consider some evaluation identities, and show that they, too, may be proved using $p$-adic representation theory.  More precisely, let $F$ be a non-archimedean local field with residual field of odd characteristic.  Let $E$ be an unramified quadratic extension of $F$.  We set up the following cases:

\begin{center}
\begin{tabular}{|c|c|c|}
\hline
Cases & $G$ & $H$ \\
\hline
\textbf{Case 1} &  $Gl_{2n}(F)$ & $Gl_{n}(E)$ \\
\hline
\textbf{Case 2} & $Gl_{2n}(E)$ & $Gl_{2n}(F)$ \\
\hline
\textbf{Case 3} & $Gl_{2n}(F)$ & $Sp_{2n}(F)$ \\
\hline
\textbf{Case 4} & $Gl_{2n}(F)$ & $Gl_{n}(F) \times Gl_{n}(F)$ \\
\hline
\end{tabular}
\end{center}

\noindent For simplicity, we will assume $F = \mathbb{Q}_{p}$ and $E = \mathbb{Q}_{p}(\sqrt{a})$, for $p$ an odd prime and $a$ prime to $p$ and without a square root.  However, our arguments used in this paper apply to \textit{any} $F,E$ satisfying the above conditions.  Let $K$ denote the maximal compact subgroup of $G$ and $K'$ the maximal compact subgroup of $H$.  In all four cases, there is an embedding of $H$ inside $G$, and an involution on $G$ that has $H$ as its set of fixed points; $S := G/H$ is the resulting $p$-adic symmetric space (see Background subsection 2 for more details).  For these symmetric spaces, relative zonal spherical functions and a Plancherel theorem were computed by Offen in \cite{Offen} and Hironaka-Sato in \cite{HS}.  The method used is that of Casselman and Shalika \cite{C, CS}, who provide another derivation of Macdonald's formula for zonal spherical functions (see \cite{MacP} for the general reductive group case) using the theory of admissible representations of $p$-adic reductive groups.  We use these works to prove the following integral identities in the third section of the paper:

\begin{theorem} \label{genthm}
Let $l(\lambda) \leq 2n$, $l(\mu) \leq n$, and $c_{\lambda} \in \mathcal{H}(G,K)$ be the characteristic function of the double coset $Kp^{\lambda}K$.  Then we have the following integral evaluations
\begin{enumerate}
\item (Case 1)
\begin{multline*}
\frac{1}{Z}\int_{T} P_{\lambda}^{(2n)}(x_{i}^{\pm 1}; p^{-1}) K_{\mu}^{BC_{n}}(x; p^{-1}; \pm p^{-1/2},0,0) \tilde \Delta_{K}^{(n)}(x;p^{-1}; \pm p^{-1/2},0,0) dT \\
= p^{\langle\mu,\rho_{1}\rangle - \langle\lambda, \rho_{2}\rangle} \frac{V_{0}}{V_{\mu}}\int_{H} c_{\lambda}(g_{\mu}h) dh
\end{multline*}
\item (Case 2)
\begin{multline*}
\frac{1}{Z}\int_{T}P_{\lambda}^{(2n)}(x_{i}^{\pm 1};p^{-2}) K_{\mu}^{BC_{n}}(x; p^{-2}; 1, p^{-1},0,0) \tilde \Delta_{K}^{(n)}(x;p^{-2};1,p^{-1},0,0)dT \\
= p^{2\langle \mu,\rho_{1} \rangle - 2 \langle \lambda, \rho_{2} \rangle}  \frac{V_{0}}{V_{\mu}}\int_{H} c_{\lambda}(g_{\mu}k_{0}h) dh
\end{multline*}
\item (Case 3)
\begin{multline*}
\frac{1}{Z}\int_{T} P_{\lambda}^{(2n)}(p^{\pm 1/2}x_{i};p^{-1})P_{\mu}^{(n)}(x^{-1};p^{-2}) \tilde \Delta_{S}^{(n)}(x;p^{-2})dT \\
= p^{\langle \mu, \rho_{3}\rangle - \langle \lambda, \rho_{2}\rangle} \frac{V_{0}}{V_{\mu}}\int_{H} c_{\lambda}(g_{-\mu}h) dh
\end{multline*}
\item (Case 4)
\begin{multline*}
\frac{1}{Z}\int_{T} P_{\lambda}^{(2n)}(x_{i}^{\pm 1}; p^{-1}) K_{\mu}^{BC_{n}}(x; p^{-1}; p^{-1/2}, p^{-1/2},0,0) \tilde \Delta_{K}^{(n)}(x;p^{-1}; p^{-1/2}, p^{-1/2},0,0) dT \\
= p^{\langle\mu,\rho_{1}\rangle - \langle\lambda, \rho_{2}\rangle} \frac{V_{0}}{V_{\mu}}\int_{H} c_{\lambda}(g_{\mu}k_{0}h) dh
\end{multline*}
\end{enumerate}
where $\rho_{1} = (n - 1/2, n-3/2, \dots, 1/2) \in \mathbb{C}^{n}$, $\rho_{2} = (n-1/2, n-3/2, \dots, 1/2 - n) \in \mathbb{C}^{2n}$, $\rho_{3} = (n-1, n-3, \dots, 1-n) \in \mathbb{C}^{n}$ and the normalization $Z$ is the evaluation of the integral at $\lambda = \mu = 0$.  Thus, when $\mu = 0$, up to a scaling factor it is equal to
\begin{equation*}
\int_{H} c_{\lambda}(h) dh.
\end{equation*}
\end{theorem}

\noindent Explicit formulas for $V_{\mu}$, $g_{\mu} \in G$, and $k_{0} \in K$ for each case may be found in Sections 2 and 3, respectively.  Also, $K_{\mu}^{BC_{n}}(x;t;a,b,c,d)$ are the Koornwinder polynomials at $q=0$; they are generalizations of symplectic and orthogonal group characters and multivariate generalizations of Askey-Wilson polynomials, see \cite{K}.  These Laurent polynomials are invariant under permutations of variables and inverting variables.  We note the symmetric function theory interpretation of the above identities: in Case 1 for example, if one expands the specialized Hall-Littlewood polynomial $P_{\lambda}^{(2n)}(x_{i}^{\pm 1};t)$ in terms of the corresponding Koornwinder basis, the integral gives the coefficient on $K_{\mu}^{BC_{n}}(x;t; \pm \sqrt{t},0,0)$ (with $t = p^{-1}$).  

In the first three cases, specializing $\mu = 0$ in the above theorem provides an interpretation of Corollary 14, Corollary 15 and Theorem 22 of \cite{VV} using $p$-adic representation theory.  It also provides a $p$-adic proof of these identities: we will evaluate the right hand side by using the Cartan decompositions for $G$ and $H$, along with some measure computations.  As a consequence, we will show that the $\mu = 0$ version of the integrals above in Case (1) and (3) vanish unless $\lambda = \nu^{2}$ for some $\nu$ (resp. $\lambda = \nu \bar{\nu}$), and if this is satisfied it is a certain $p$-adic measure count.  In Case (2), we obtain an explicit integral evaluation at $\mu = 0$, again in terms of a rational function arising from a certain $p$-adic quantity.  The last case, Case (4), provides a new integral evaluation at $\mu = 0$, which is a $t$-generalization of Theorem \ref{symplecticgp} part (1) (in a different direction than Case (1), which is also a $t$-generalization of the same result).  One can refer to Theorem \ref{new} within the paper for details.

We remark that the above integrals at $\mu = 0$ have a direct application to Littlewood summation identities.  Indeed, as demonstrated in \cite{VV}, one can start with the integral identities and use a procedure of \cite{R} to show that in the $n \rightarrow \infty$ limit, one obtains a Littlewood summation identity.  In particular, we do this for Case (4) in this paper and obtain a new Littlewood summation identity:

\begin{theorem} \label{Littlewood}
The following formal identity holds:
\begin{equation*}
\sum_{\mu, \nu} P_{\mu \cup \nu}(x;t) t^{\langle \mu + \nu, \rho \rangle - \frac{1}{2}\langle \mu \cup \nu, \rho \rangle} \frac{b_{\mu \cup \nu}(t)}{b_{\mu}(t) b_{\nu}(t)} = \prod_{j<k} \frac{1-tx_{j}x_{k}}{1-x_{j}x_{k}} \prod_{j} \frac{1+\sqrt{t}x_{j}}{1-\sqrt{t}x_{j}},
\end{equation*}
where $\rho = (1,3,5, \dots)$.
\end{theorem}

\noindent One can refer to Section 2 for the definition of $b_{\lambda}(t), \mu \cup \nu$ etc.  This is a $t$-generalization of the classical Littlewood identity for Schur functions (see \cite{L}, \cite{Mac}):
\begin{equation*}
\sum_{\substack{\lambda \\ \lambda' \text{ even }}} s_{\lambda}(x) = \prod_{i<j} \frac{1}{1-x_{i}x_{j}}.
\end{equation*}

Our method of using integration over $p$-adic groups supports a generalization of the usual vanishing identities at the Hall-Littlewood level; this is provided in the second part of the paper, and we briefly detail it here.  Note that Theorem \ref{symplecticgp} part (1) provides the coefficient on the trivial character when one expands the restricted Schur function in the basis of symplectic characters.  A natural question, then, is whether there are interesting vanishing conditions for the other coefficients in this expansion.  This is addressed by the following classical result:
\begin{theorem} (Littlewood and Weyl) \label{WL}
If $l(\lambda) \leq n$, we have the branching rule
\begin{equation*}
s_{\lambda}^{(2n)}(x^{\pm 1}) = \sum_{l(\mu) \leq n} sp_{\mu}(x_{1}, \dots, x_{n}) \Bigg( \sum_{\substack{\beta \in \Lambda_{2n}^{+} \\ \beta = \nu^{2}}} c^{\lambda}_{\mu, \beta} \Bigg),
\end{equation*}
where $c^{\lambda}_{\mu, \beta}$ are the Littlewood-Richardson coefficients and $sp_{\mu}$ is an irreducible symplectic character.  
\end{theorem}
\noindent One may rephrase this in terms of an integral identity as follows: for $l(\lambda) \leq n$, the integral
\begin{equation*}
\int_{S \in Sp(2n)} s_{\lambda}(S) sp_{\mu}(S) dS
\end{equation*}
vanishes if and only if $c^{\lambda}_{\mu, \beta} = 0$ for all $\beta = \nu^{2} := (\nu_{1}, \nu_{1}, \nu_{2}, \nu_{2}, \dots) \in \Lambda_{2n}^{+}$.  We prove the following $t$-analog of this result:
\begin{theorem} \label{genvan}  Let $\lambda, \mu \in \Lambda_{n}^{+}$.  Then the following three statements are equivalent:
\begin{enumerate}
\item The generalized integral
\begin{equation*}
\frac{1}{\int_{T} \tilde \Delta_{K}^{(n)}(x; \pm \sqrt{t},0,0;t) dT} \int_{T} P_{\lambda}^{(2n)}(x_{i}^{\pm 1};t) K_{\mu}^{BC_{n}}(x;t; \pm \sqrt{t},0,0) \tilde \Delta_{K}^{(n)}(x; \pm \sqrt{t},0,0;t)dT
\end{equation*}
vanishes as a rational function of $t$. 
\item The Hall polynomials
\begin{equation*}
g^{\lambda}_{\mu, \beta}(t) 
\end{equation*}
vanish as a function of $t$, for all $\beta \in \Lambda_{2n}^{+}$ with $\beta = \nu^{2}$ for some $\nu$.
\item The Littlewood-Richardson coefficients 
\begin{equation*}
c^{\lambda}_{\mu, \beta}
\end{equation*}
are equal to $0$ for all $\beta \in \Lambda_{2n}^{+}$ with $\beta = \nu^{2}$ for some $\nu$.
\end{enumerate}
\end{theorem}
\noindent The proof of this relies on $p$-adic arguments similar to those used to prove Theorem \ref{genthm}, as well as some technical arguments involving Hall polynomials.  We also provide some examples following the proof of Theorem \ref{genvan} that makes use of known information about the vanishing of the Littlewood-Richardson coefficients (or Hall polynomials).

In the first section of the paper, we'll review some relevant background and notation.  In the second section, we'll give an interpretation of some identities in \cite{VV} using $p$-adic representation theory that may be viewed as an analog of the Schur identities.  Finally, in the last section, we'll use this approach to provide a generalization of the integral identities considered in this paper.

\textbf{Acknowledgements:} The author would like to thank E. Rains for initially suggesting this problem and for many helpful conversations along the way.  She would also like to thank A. Borodin and M. Vazirani for useful discussions.  She also thanks O. Warnaar for pointing out a typo in the density computation in Case 4 in an earlier version of this article.

\section{Background}
\subsection{Symmetric function theory}
We first review the relevant notations that we will be using in this paper.  

Let 
\begin{equation*}
\Lambda_{n}^{+} = \{ (\lambda_{1}, \dots, \lambda_{n}) \in \mathbb{Z}^{n}| \lambda_{1} \geq \cdots \geq \lambda_{n} \geq 0 \}
\end{equation*}
be the set of \textit{partitions}.  We will refer to the $\lambda_{i}$ as the \textit{parts} of $\lambda$.  The \textit{length} of a partition $\lambda$ is the number of nonzero $\lambda_{i}$.  Also let $m_{i}(\lambda)$ be the number of $\lambda_{j}$ equal to $i$ for each $i \geq 0$; this is the \textit{multiplicity} of the part $i$ in $\lambda$.  We will write $\lambda = \mu^{2}$ if $\lambda = (\mu_{1}, \mu_{1}, \mu_{2}, \mu_{2}, \dots)$ and we will say $\lambda$ has all parts occuring with even multiplicity.  Note that this is the same as the conjugate partition $\lambda'$ having all even parts.  For example, if $\lambda = (3,3,2,2,1,1,1,1)$ then $\lambda = \mu^{2}$ with $\mu = (3,2,1,1)$.

Following \cite{Mac}, we define $\lambda \cup \mu$ to be the partition whose parts are those of $\lambda$ and $\mu$, arranged in descending order.  For example, if $\lambda = (3,2,1)$ and $\mu = (2,2)$ then $\lambda \cup \mu = (3,2,2,2,1)$.  Also $\langle \lambda, \mu \rangle = \lambda \cdot \mu = \lambda_{1}\mu_{1} + \cdots + \lambda_{n}\mu_{n}$.

We define
\begin{equation*}
\rho_{1} = (n - 1/2, n-3/2, \dots, 1/2) \in \mathbb{C}^{n}
\end{equation*}
\begin{equation*}
\rho_{2} = (n-1/2, n-3/2, \dots, 1/2 - n) \in \mathbb{C}^{2n}
\end{equation*}
\begin{equation*}
\rho_{3} = (n-1, n-3, \dots, 1-n) \in \mathbb{C}^{n};
\end{equation*}
which will be used throughout the paper.

The symmetric $q=0$ Macdonald density \cite{RV} is
\begin{align} \label{Sd}
\tilde \Delta_{S}^{(n)}(x;t) &= \frac{1}{n!}\prod_{1 \leq i \neq j \leq n} \frac{1-x_{i}x_{j}^{-1}}{1-tx_{i}x_{j}^{-1}}
\end{align}
and the symmetric $q=0$ Koornwinder density \cite{K} is
\begin{equation} \label{Kd}
\tilde \Delta_{K}^{(n)}(x;t;a,b,c,d) = \frac{1}{2^{n}n!} \prod_{1 \leq i \leq n} \frac{1-x_{i}^{\pm 2}}{(1-ax_{i}^{\pm 1})(1-bx_{i}^{\pm 1})(1-cx_{i}^{\pm 1})(1-dx_{i}^{\pm 1})} \prod_{1 \leq i<j \leq n} \frac{1-x_{i}^{\pm 1}x_{j}^{\pm 1}}{1-tx_{i}^{\pm 1}x_{j}^{\pm 1}},
\end{equation}
where we write $1-x_{i}^{\pm 2}$ for the product $(1-x_{i}^{2})(1-x_{i}^{-2})$ and $1-x_{i}^{\pm 1}x_{j}^{\pm 1}$ for $(1-x_{i}x_{j})(1-x_{i}^{-1}x_{j}^{-1})(1-x_{i}^{-1}x_{j})(1-x_{i}x_{j}^{-1})$ etc. 

Hall-Littlewood polynomials $P_{\lambda}^{(n)}(x; t)$ indexed by partitions $\lambda$ with length at most $n$ form an orthogonal basis with respect to (\ref{Sd}).  Similarly, Koornwinder polynomials $K_{\mu}^{(n)}(x;t;a,b,c,d)$ indexed by partitions $\mu$ with length at most $n$ form an orthogonal basis with respect to (\ref{Kd}).  An explicit formula for $K_{\mu}^{(n)}(x;t;a,b,c,d)$ was provided in \cite{VV2}.

We also define
\begin{equation*}
v_{m}(t) = \prod_{i=1}^{m} \frac{1-t^{i}}{1-t} = \frac{\phi_{m}(t)}{(1-t)^{m}}.
\end{equation*}
We use this to define
\begin{align*}
v_{\lambda}(t) &= \prod_{i \geq 0} v_{m_{i}(\lambda)}(t);
\end{align*}
this is the factor that makes the Hall-Littlewood polynomials of type A monic.  Also let
\begin{equation*}
b_{\lambda}(t) = \prod_{i \geq 1} \phi_{m_{i}(\lambda)}(t).
\end{equation*}

Finally, for ease of notation, we set-up some shorthand notation for some specific densities that will be used throughout the paper.  Let $p$ be a prime, and define
\begin{equation} \label{den1}
\Delta_{1} = \tilde \Delta_{K}^{(n)}(x; p^{-1}; \pm p^{-1/2},0,0)
\end{equation}  
\begin{equation} \label{den2}
\Delta_{2} = \tilde  \Delta_{K}^{(n)}(x;p^{-2};1,p^{-1},0,0)
\end{equation}
\begin{equation} \label{den3}
\Delta_{3} = \tilde \Delta_{S}^{(n)}(x;p^{-2})
\end{equation}
\begin{equation} \label{den4}
\Delta_{4} = \tilde \Delta_{K}^{(n)}(x;p^{-1};  p^{-1/2}, p^{-1/2},0,0). 
\end{equation}

\subsection{$P$-adic represention theory}
Let $F$ be a non-archimedean local field with residual field of odd characteristic.  Let $E$ be an unramified quadratic extension of $F$.  We set up the following cases corresponding to the above theorems:

\textbf{Case 1:} $G = Gl_{2n}(F), H = Gl_{n}(E)$ 

\textbf{Case 2:} $G = Gl_{2n}(E), H = Gl_{2n}(F)$

\textbf{Case 3:} $G = Gl_{2n}(F), H = Sp_{2n}(F)$

\textbf{Case 4:} $G = Gl_{2n}(F), H = Gl_{n}(F) \times Gl_{n}(F)$.

For simplicity, from now on we will assume $F = \mathbb{Q}_{p}$ and $E = \mathbb{Q}_{p}(\sqrt{a})$, for $p$ an odd prime and $a$ prime to $p$ and without a square root.  However, the argument applies to \textit{any} $F,E$ satisfying the above conditions.  The number of elements in the residual field of $F$ is $p$, and for $E$ it is $p^{2}$.  Throughout, we will use $K$ to denote the maximal compact subgroup of $G$ (so $K=Gl_{2n}(\mathbb{Z}_{p})$ in Cases 1 and 3 and $K = Gl_{2n}(\mathbb{Z}_{p}(\sqrt{a}))$ in Case 2) and $K'$ the maximal compact subgroup of $H$.  

Define
\begin{equation*}
\Lambda_{n}^{+} = \{ \lambda = (\lambda_{1}, \dots, \lambda_{n}) \in \mathbb{Z}^{n} | \lambda_{1} \geq \cdots \geq \lambda_{n} \geq 0 \}
\end{equation*}
and
\begin{equation*}
\Lambda_{n} = \{ \lambda = (\lambda_{1}, \dots, \lambda_{n}) \in \mathbb{Z}^{n} | \lambda_{1}  \geq \cdots \geq \lambda_{n} \},
\end{equation*}
so that $\Lambda_{n}^{+}$ is the set of partitions with length at most $n$, while $\lambda \in \Lambda_{n}$ is allowed to have negative parts.

We set-up some notation following \cite{Offen}, \cite{HS}.  Let $g \mapsto g^*$ denote the involution on $G$ given by:

\noindent \textbf{Case 1:} $g^* = g^{-1}$.

\noindent \textbf{Case 2:} $g^* = \bar g^{-1}$.

\noindent \textbf{Case 3:} $g^* = g^t$.

\noindent \textbf{Case 4:} $g^* = \epsilon g^{-1} \epsilon$, where
\begin{equation*}
\epsilon = \left( \begin{array}{cc}
I_{n} & 0   \\
0 & -I_{n}  \\  \end{array} \right) \in G
\end{equation*}

Fix the element $s_{0} \in G$ to be
\begin{equation*} 
s_0 = \begin{cases}
\left( \begin{array}{cc}
0 & w_{n}   \\
aw_{n} & 0  \\  \end{array} \right), & \textrm{Case 1}\\
I_{2n}, & \textrm{Case 2}\\
J_n, & \textrm{Case 3} \\
I_{2n}, & \textrm{Case 4}
\end{cases},
\end{equation*}
where $J_n = (\begin{smallmatrix} 0 & I_n\\-I_n & 0\end{smallmatrix})$ and $w_{n}$ is the $n \times n$ matrix with ones on the anti-diagonal and zeroes everywhere else.  Define $S = G \cdot s_{0}$, where the action is $g \cdot s_{0} = gs_{0}g^{*}$.  In all cases, $H$ may be identified with the stabilizer of $s_{0}$ in $G$:

 \noindent In \textbf{Case 1}:
\begin{equation*}
Stab(s_{0}) = \Bigg\{ \left( \begin{array}{cc}
i & j   \\
aw_{n}jw_{n} & w_{n}iw_{n}  \\  \end{array} \right) \in G | i,j \in Gl_{n}(\mathbb{Q}_{p}) \Bigg \} \cong Gl_{n}(\mathbb{Q}_{p}(\sqrt{a})). 
\end{equation*}
\noindent In \textbf{Case 2}:
\begin{equation*}
 Stab(s_{0}) = GL_{2n}(F).
\end{equation*}
\noindent In \textbf{Case 3}:
\begin{equation*}
 Stab(s_{0}) = Sp_{2n}(F)
\end{equation*}
\noindent In \textbf{Case 4}:
\begin{equation*}
Stab(s_{0}) = \Bigg\{ \left( \begin{array}{cc}
g_{1} & 0   \\
0 & g_{2}  \\  \end{array} \right) \in G | g_{1},g_{2} \in Gl_{n}(\mathbb{Q}_{p}) \Bigg \} \cong Gl_{n}(\mathbb{Q}_{p}) \times Gl_{n}(\mathbb{Q}_{p}).
\end{equation*}

In Case 1, we identify the maximal compact subgroup $K' = Gl_{n}(\mathbb{Z}_{p}(\sqrt{a})) \subset Gl_{n}(\mathbb{Q}_{p}(\sqrt{a}))$ with $K \cap H \subset Gl_{2n}(\mathbb{Q}_{p})$.  In the other two cases, $K' = Gl_{2n}(\mathbb{Z}_{p})$ and $Sp_{2n}(\mathbb{Z}_{p})$, respectively.  (In all cases, we identify $K'$ (the maximal compact subgroup of $H$) with its image in $G$ under the corresponding embedding.) The map $\theta: G \rightarrow S$ defined by $\theta(g) = gs_{0}g^{*} = g \cdot s_{0}$ induces a bijection between $G/H$ and $S$.   

Let $\mathcal{H}(G,K)$ be the Hecke algebra of $G$ with respect to $K$; it is the convolution algebra of compactly supported, $K$-bi-invariant, complex valued functions on $G$.  Let $C^{\infty}(K\setminus S)$ be the space of $K$-invariant complex valued functions on $S$.  Also define $\mathcal{S}(K \setminus S)$ to be the $\mathcal{H}(G,K)$-submodule of $K$-invariant functions on $S$ with compact support.
Define a $\mathcal{H}(G,K)$-module structure on $C^{\infty}(K \setminus S)$ via the convolution operation
\begin{align*}
f \star \phi(s) = \int_{G} f(g) \phi(g^{-1} \cdot s) dg
\end{align*}
where $f \in \mathcal{H}(G,K)$ and $\phi \in C^{\infty}(K \setminus S)$ and $dg$ is the Haar measure on $G$ normalized so $\int_{K} dg = 1$.  

\begin{definition} \cite{Offen, HS}
A \textbf{relative spherical function} on $S$ is an eigenfunction $\Omega \in C^{\infty}(K \setminus S)$ of $\mathcal{H}(G,K)$ under this convolution, normalized so that $\Omega(s_{0}) = 1$.  
\end{definition}

Define the elements $d_\lambda$ in $G$ as follows:

\noindent\textbf{Case 1}:
\begin{equation*}
d_{\lambda} = \text{antidiag}.(p^{\lambda_{1}}, \dots, p^{\lambda_{n}}, ap^{-\lambda_{n}}, \dots, ap^{-\lambda_{1}}).
\end{equation*}
\noindent\textbf{Case 2}:
\begin{equation*}
 d_{\lambda} = \text{antidiag}.(p^{\lambda_1},\dotsc,p^{\lambda_n},p^{-\lambda_n},\dotsc,p^{-\lambda_1}).
\end{equation*}
\noindent\textbf{Case 3}:
\begin{equation*}
 d_\lambda = \text{antidiag}.(p^{\lambda_1},\dotsc,p^{\lambda_n},-p^{\lambda_n},\dotsc,-p^{\lambda_1}).
\end{equation*}
\noindent\textbf{Case 4}:
\begin{equation*}
d_{\lambda} = \text{antidiag}.(p^{\lambda_1},\dotsc,p^{\lambda_n},-p^{-\lambda_n},\dotsc,-p^{-\lambda_1}).
\end{equation*}

Note that $d_{0} = s_{0}$ in Cases 1 and 3.  

\begin{theorem} \cite[Proposition 3.1]{Offen}, \cite{HS}
The $K$-orbits of $S$ are given by the disjoint union
\begin{equation*}
S = \cup K \cdot d_{\lambda},
\end{equation*}
varying over $\lambda \in \Lambda_{n}^{+}$.  
\end{theorem}
Let $ch_{\lambda}$ denote the characteristic function for the $K$-orbit $K \cdot d_{\lambda}$, then the space $\mathcal{S}(K \setminus S)$ is spanned by the functions $\{ ch_{\lambda} | \lambda \in \Lambda_{n}^{+} \}$.  

\begin{theorem} 
(Cartan decomposition for $G$, see \cite{Mac} for example)  We have the disjoint union
\begin{equation*}
G = \cup Kp^{\lambda}K,
\end{equation*}
varying over $\lambda \in \Lambda_{2n}$, and $p^{\lambda}$ denotes the diagonal matrix $\text{diag}.(p^{\lambda_{1}}, p^{\lambda_{2}}, \dots, p^{\lambda_{2n}})$.  
\end{theorem}

Let $c_{\lambda}$, with $\lambda \in \Lambda_{2n}$, be the characteristic function for the double coset $Kp^{\lambda}K$ inside $G$.  These functions form a basis for $\mathcal{H}(G,K)$.  

Let the constant $V_{\lambda}$ (for any $l(\lambda) \leq n$) be given via the following integral evaluations (see Equations \ref{den1}, \ref{den2}, \ref{den3}, \ref{den4} for notation):

In \textbf{Case 1:}
\begin{equation*}
\frac{1}{V_{\lambda}} = \int_{T} K_{\lambda}^{BC_{n}}(x;p^{-1}; \pm p^{-1/2},0,0)^{2} \Delta_{1} dT
\end{equation*}

In \textbf{Case 2:}
\begin{equation*}
\frac{1}{V_{\lambda}} = \int_{T} K_{\lambda}^{BC_{n}}(x;p^{-2};1,p^{-1},0,0)^{2} \Delta_{2} dT
\end{equation*}

In \textbf{Case 3:}
\begin{equation*}
\frac{1}{V_{\lambda}} = \int_{T} P_{\lambda}^{(n)}(x;p^{-2}) P_{\lambda}^{(n)}(x^{-1};p^{-2}) \Delta_{3} dT
\end{equation*}

In \textbf{Case 4:}
\begin{equation*}
\frac{1}{V_{\lambda}} = \int_{T} K_{\lambda}^{BC_{n}}(x;p^{-1}; p^{-1/2}, p^{-1/2},0,0)^{2} \Delta_{4} dT
\end{equation*}

In particular, $V_{0}$ is the reciprocal of the integral of the density function in each of the cases.  
Note that for Cases 1, 2, and 4, the $V_{\lambda}$ are determined explicitly in \cite{VV2} for general parameters $t_{0}, \dots, t_{3}$ of the Koornwinder $q=0$ polynomials and in \cite{MacP} for these choices of parameters.  In the third case, the norm is computed in \cite{VV}, for example.

Finally, for $z = (z_{1}, \dots, z_{n}) \in \mathbb{C}^{n}$ and $f \in \mathcal{H}(G,K)$, define (for \textbf{Case 1}, \textbf{Case 2}, and \textbf{Case 4})
\begin{equation*}
\tilde f(z) = \hat{f}(z_{1}, \dots, z_{n}, -z_{1}, \dots, -z_{n})
\end{equation*}
where $\hat{}$ denotes the Satake transform on $\mathcal{H}(G,K)$.  For \textbf{Case 3}, define  
\begin{equation*}
\tilde f(z) = \hat{f}(z_{1} + 1/2, z_{1} - 1/2, \dots, z_{n}+1/2, z_{n} -1/2)
\end{equation*}

In fact by \cite[Lemma 4.2]{Offen} and \cite[Lemma 2.1]{HS}, $f \rightarrow \tilde f(z)$ is the eigenvalue map, that is
\begin{equation*}
(f \star \Omega_{z})(s) = \tilde f(z) \Omega_{z}(s),
\end{equation*}
where $\Omega_{z}(s)$, $z \in \mathbb{C}^{n}$ are the relative spherical functions, as determined in \cite{Offen} and \cite{HS}.

Also for $f \in \mathcal{H}(G,K), g \in G$ put $\check f(g) := f(g^{-1})$.

\section{Main Results}

In this section, we will prove Theorem \ref{genthm} and discuss the related corollaries.  The technique is to use the results of \cite{Offen, HS} as well as additional $p$-adic arguments.

\begin{proposition}\label{prop1}
Let $l(\lambda) \leq 2n$ and $l(\mu) \leq n$.  Then we have
\begin{multline*}
\int_{S} ( c_{\lambda} \star ch_{0})(s) ch_{\mu}(s) ds \\
=\begin{cases} \frac{ p^{\langle\mu,\rho_{1}\rangle + \langle\lambda, \rho_{2}\rangle}}{Z} \int_{T} P_{\lambda}^{(2n)}(x_{i}^{\pm 1}; p^{-1}) K_{\mu}^{BC_{n}}(x; p^{-1}; \pm p^{-1/2},0,0) \Delta_{1} dT, & \text{in \textbf{Case 1}} \\
 \frac{p^{2\langle \mu,\rho_{1} \rangle + 2 \langle \lambda, \rho_{2} \rangle}}{Z} \int_{T}P_{\lambda}^{(2n)}(x_{i}^{\pm 1};p^{-2}) K_{\mu}^{BC_{n}}(x; p^{-2}; 1, p^{-1},0,0) \Delta_{2} dT ,& \text{in \textbf{Case 2}}\\
\frac{p^{\langle \mu, \rho_{3}\rangle + \langle \lambda, \rho_{2}\rangle}}{Z} \int_{T} P_{\lambda}^{(2n)}(p^{\pm 1/2}x_{i};p^{-1})P_{\mu}^{(n)}(x^{-1};p^{-2}) \Delta_{3} dT,& \text{in \textbf{Case 3}}\\
\frac{p^{\langle \mu, \rho_{1} \rangle + \langle \lambda, \rho_{2} \rangle}}{Z}\int_{T} P_{\lambda}^{(2n)}(x_{i}^{\pm 1}; p^{-1}) K_{\mu}^{BC_{n}}(x; p^{-1};  p^{-1/2},p^{-1/2},0,0) \Delta_{4} dT, & \text{in \textbf{Case 4}},
\end{cases}
\end{multline*}
where $\rho_{1} = (n - 1/2, n-3/2, \dots, 1/2) \in \mathbb{C}^{n}$, $\rho_{2} = (n-1/2, n-3/2, \dots, 1/2 - n) \in \mathbb{C}^{2n}$, $\rho_{3} = (n-1, n-3, \dots, 1-n) \in \mathbb{C}^{n}$ and the normalization $Z$ is the evaluation of the integral at $\lambda = \mu = 0$.
\end{proposition}
\begin{proof}
\hfill

\textbf{Case 1.} We use the spherical Fourier transform on $\mathcal{S}(K \setminus S)$:
\begin{equation*}
\int_{S} f_{1}(s) \overline{f_{2}(s)} ds = \int_{T} \hat{f_{1}}(z) \overline{\hat{f_{2}}(z)} d_{\mu}(z);
\end{equation*}
here $d_{\mu}(z)$ is the Plancherel measure on $\mathcal{S}(K \setminus S)$.
We apply this to 
\begin{equation*}
\int_{S} (c_{\lambda} \star ch_{0})(s) ch_{\mu}(s) ds.
\end{equation*}

Note that the spherical Fourier transform satisfies (by Lemma 4.4 \cite{Offen})
\begin{equation*}
(c_{\lambda} \star ch_{0}) \hat (z) =  \tilde c_{\lambda}(z) \hat{ch}_{0}(z) = \tilde c_{\lambda}(z),
\end{equation*}
since $\hat{ch}_{0}(s) = 1$.  
Here
\begin{equation*}
\tilde c_{\lambda}(z) = \hat{c_{\lambda}}(z_{1}, \dots, z_{n}, -z_{1}, \dots, -z_{n})
\end{equation*}
where $\hat{c_{\lambda}}$ denotes here the usual Satake transform on $\mathcal{H}(Gl_{2n}(\mathbb{Q}_{p}), Gl_{2n}(\mathbb{Z}_{p}))$.  But \cite[Ch. V]{Mac} , this is equal to 
\begin{equation*}
p^{\langle\lambda, \rho_{2}\rangle} P_{\lambda}^{(2n)}(p^{-z_{1}}, \dots, p^{-z_{n}}, p^{z_{1}}, \dots, p^{z_{n}}; p^{-1}).
\end{equation*}

Also, using \cite{Offen} Theorem 1.2 and Proposition 5.15, we have
\begin{multline*}
\hat{ch_{\mu}}(z) = \Big\{ \int_{K \cdot d_{\mu}} ds \Big\} \Omega_{z}(d_{\mu}) = \Big \{ p^{2\langle\mu, \rho_{1}\rangle} \frac{V_{0}}{V_{\mu}} \Big \} p^{-\langle\mu, \rho_{1}\rangle} \frac{V_{\mu}}{V_{0}} K_{\mu}^{BC_{n}}(p^{z_{i}}; p^{-1}; \pm p^{-1/2},0,0) \\
= p^{\langle\mu, \rho_{1}\rangle} K_{\mu}^{BC_{n}}(p^{z_{i}}; p^{-1}; \pm p^{-1/2},0,0).
\end{multline*}

Finally, by \cite{Offen} Theorem 1.3 the Plancherel density is 
\begin{equation*}
\frac{\tilde \Delta_{K}^{(n)}(p^{z_{i}};p^{-1}; \pm p^{-1/2},0,0 )}{\int_{T} \tilde \Delta_{K}^{(n)}(p^{z_{i}};p^{-1}; \pm p^{-1/2},0,0)dT}.
\end{equation*}
Combining these, using Equation \ref{den1}, and putting $x_{i} = p^{z_{i}}$ gives the result.

\textbf{Case 2.} The argument is the same as Case 1, but the Plancherel measure and zonal spherical functions are different.  We indicate the differences (see the above references of \cite{Offen} but for Case 2, and \cite[Ch. V]{Mac} for the group $Gl_{2n}(E)$):
\begin{equation*}
\tilde c_{\lambda}(z) = \hat{c_{\lambda}}(z_{1}, \dots, z_{n}, -z_{1}, \dots, -z_{n}) = p^{2\langle \lambda, \rho_{2}  \rangle} P_{\lambda}^{(2n)} (p^{-2z_{1}}, \dots, p^{-2z_{n}}, p^{2z_{1}}, \dots, p^{2z_{n}}; p^{-2}).
\end{equation*}
We also have
\begin{multline*}
\hat{ch_{\mu}}(z) =  \Big\{ \int_{K \cdot d_{\mu}} ds \Big\} \Omega_{z}(d_{\mu}) = \Big \{ p^{4 \langle \mu, \rho_{1} \rangle} \frac{V_{0}}{V_{\mu}} \Big \} p^{-2\langle \mu, \rho_{1} \rangle} \frac{V_{\mu}}{V_{0}} K_{\mu}^{BC_{n}}(p^{2z_{i}}; p^{-2}; 1,p^{-1},0,0)  \\
= p^{2 \langle \mu, \rho_{1} \rangle}K_{\mu}^{BC_{n}}(p^{2z_{i}}; p^{-2}; 1,p^{-1},0,0).
\end{multline*}
Finally, the Plancherel density is
\begin{equation*}
\frac{\tilde \Delta_{K}^{(n)}(p^{2z_{i}}; p^{-2}; 1, p^{-1},0,0)}{\int_{T}\tilde \Delta_{K}^{(n)}(p^{2z_{i}}; p^{-2}; 1, p^{-1},0,0) dT}.
\end{equation*}
Combining these, and putting $x_{i} = p^{2z_{i}}$ gives the result.

\textbf{Case 3.} The argument is the same as in the above cases, but the Plancherel measure and zonal spherical functions are different.  We indicate the differences (see \cite{HS}):
\begin{multline*}
\tilde c_{\lambda}(z) = \hat{c_{\lambda}}(z_{1} + 1/2, z_{1} - 1/2, \dots, z_{n} + 1/2, z_{n} - 1/2) \\
= p^{\langle \lambda, \rho_{2} \rangle} P_{\lambda}^{(2n)}(p^{-z_{1} - 1/2}, p^{-z_{1} + 1/2}, \dots, p^{-z_{n} - 1/2}, p^{-z_{n} + 1/2}; p^{-1}).
\end{multline*}
We also have
\begin{multline*}
\hat{ch_{\mu}}(z) =  \Big\{ \int_{K \cdot d_{\mu}} ds \Big\} \Omega_{z}(d_{\mu}) = \Big \{ p^{2\langle \mu, \rho_{3} \rangle} \frac{V_{0}}{V_{\mu}} \Big\} p^{-<\mu, \rho_{3}>} \frac{V_{\mu}}{V_{0}} P_{\mu}(p^{z_{1}}, \dots, p^{z_{n}}; p^{-2}) \\
= p^{\langle \mu, \rho_{3} \rangle} P_{\mu}(p^{z_{1}}, \dots, p^{z_{n}}; p^{-2}). 
\end{multline*}
Finally, the Plancherel density is
\begin{equation*}
\frac{\tilde \Delta_{S}^{(n)}(p^{z_{i}};p^{-2})}{\int_{T}\tilde \Delta_{S}^{(n)}(p^{z_{i}};p^{-2}) dT }.
\end{equation*}
Combining these, and putting $x_{i} = p^{z_{i}}$ gives the result.

\textbf{Case 4.} The argument is the same as Case 1, but the Plancherel measure and zonal spherical functions are different.  We indicate the differences (see the above references of \cite{Offen} but for Case 1, as well as \cite{Mac}):
\begin{equation*}
\tilde c_{\lambda}(z) = p^{\langle\lambda, \rho_{2}\rangle} P_{\lambda}^{(2n)}(p^{-z_{1}}, \dots, p^{-z_{n}}, p^{z_{1}}, \dots, p^{z_{n}}; p^{-1}).
\end{equation*}
We also have
\begin{multline*}
\hat{ch_{\mu}}(z) =  \Big\{ \int_{K \cdot d_{\mu}} ds \Big\} \Omega_{z}(d_{\mu}) = \Big \{ p^{2\langle\mu, \rho_{1}\rangle} \frac{V_{0}}{V_{\mu}} \Big \} p^{-\langle \mu, \rho_{1} \rangle} \frac{V_{\mu}}{V_{0}} K_{\mu}^{BC_{n}}(p^{z_{i}}; p^{-1}; p^{-1/2},p^{-1/2},0,0)  \\
= p^{ \langle \mu, \rho_{1} \rangle}K_{\mu}^{BC_{n}}(p^{z_{i}}; p^{-1}; p^{-1/2},p^{-1/2},0,0).
\end{multline*}
Finally, the Plancherel density is
\begin{equation*}
\frac{\tilde \Delta_{K}^{(n)}(p^{z_{i}}; p^{-1}; p^{-1/2}, p^{-1/2},0,0)}{\int_{T}\tilde \Delta_{K}^{(n)}(p^{z_{i}}; p^{-1}; p^{-1/2}, p^{-1/2},0,0) dT}.
\end{equation*}
Combining these, and putting $x_{i} = p^{z_{i}}$ gives the result.

\end{proof}

\begin{definition}
Let $g_{\mu} = \text{diag}(1, \dots, 1, p^{-\mu_{n}}, \dots, p^{-\mu_{1}}) \in Gl_{2n}(\mathbb{Q}_{p})$ for $\mu = (\mu_{1}, \dots, \mu_{n})$.  
\end{definition}

\begin{proposition}\label{prop2}
We have 
\begin{equation*}
\int_{S} (\check c_{\lambda} \star ch_{0})(s) ch_{\mu}(s) ds = \begin{cases} p^{2\langle\mu, \rho_{1}\rangle} \frac{V_{0}}{V_{\mu}}\int_{H} c_{\lambda}(g_{\mu}h) dh, & \text{in \textbf{Case 1}} \\
p^{4\langle \mu, \rho_{1}  \rangle} \frac{V_{0}}{V_{\mu}}\int_{H} c_{\lambda}(g_{\mu}k_{0}h) dh, & \text{in \textbf{Case 2}} \\
p^{2\langle \mu, \rho_{3}\rangle}\frac{V_{0}}{V_{\mu}}\int_{H} c_{\lambda}(g_{-\mu}h) dh, & \text{in \textbf{Case 3}} \\
p^{2\langle\mu, \rho_{1}\rangle} \frac{V_{0}}{V_{\mu}}\int_{H} c_{\lambda}(g_{\mu}k_{0}h) dh, & \text{in \textbf{Case 4}},
\end{cases}
\end{equation*}
where $k_{0} \in K$ is a specific element in Cases 2,4.  In particular, when $\mu = 0$, the RHS is $\int_{H} c_{\lambda}(g) dg$.
\end{proposition}
\begin{proof}
\hfill

\textbf{Case 1:}
We first note that $d_{0} = s_{0}$ in this case.  We have
\begin{equation*}
\int_{S} ( c_{\lambda} \star ch_{0})(s) ch_{\mu}(s) ds = \int_{K \cdot d_{\mu}} ({c_{\lambda}} \star ch_{0})(s) = meas.(K \cdot d_{\mu}) ({c_{\lambda}} \star ch_{0})(d_{\mu}),
\end{equation*}
where the first equality follows since $ch_{\mu}(s)$ vanishes off of $K \cdot d_{\mu}$, and the second follows since $({c_{\lambda}} \star ch_{0})$ is $K$-invariant.  Now by definition of the convolution action, we have
\begin{equation*}
({c_{\lambda}} \star ch_{0})(d_{\mu}) = \int_{G}  c_{\lambda}(g^{-1}) ch_{0}(g \cdot d_{\mu})dg.
\end{equation*}
Letting $H_{\mu} = \{g \in G | g \cdot d_{0} = d_{\mu}\}$, we have 
\begin{equation*}
g \cdot d_{\mu} \in K \cdot d_{0} \Leftrightarrow (kg) \cdot d_{\mu} = d_{0} \text{ for some $k \in K$} \Leftrightarrow g \in K H_{\mu}^{-1}.
\end{equation*}
Now one can check that $g_{\mu} \cdot d_{0} = d_{\mu}$, so that $H_{\mu} = g_{\mu} H$ (clearly $g_{\mu}H \subset H_{\mu}$, for the other direction let $g \in H_{\mu}$ then $g \cdot d_{0} = d_{\mu} = g_{\mu} \cdot d_{0}$, so $g_{\mu}^{-1}g \in H$) and so $KH_{\mu}^{-1} = KHg_{\mu}^{-1}$.  Thus, the above integral can be rewritten as
\begin{equation*}
\int_{KHg_{\mu}^{-1}} c_{\lambda}(g^{-1})dg = \int_{KH} c_{-\lambda}(gg_{\mu}^{-1}) dg.
\end{equation*}
Finally, write 
\begin{equation*}
KH = \cup K x_{i},
\end{equation*}
a disjoint union and $x_{i} \in H$.  Then we claim $H = \cup K' x_{i}$, again a disjoint union.  That the union is contained inside $H$ is clear, suppose next that $h \in H$.  But then $h = kx_{i}$ for some $k \in K$ and $x_{i}$.  But since $h, x_{i} \in H$ we have $k \in H$, i.e., $k \in K'$.  Clearly the union is disjoint, since $K'x_{i} \subset Kx_{i}$ for all $i$.  Thus,
\begin{multline*}
 \int_{KH} c_{-\lambda}(gg_{\mu}^{-1}) dg = \sum_{x_{i}} \int_{Kx_{i}} c_{-\lambda}(gg_{\mu}^{-1})dg = \sum_{x_{i}} c_{-\lambda}(x_{i}g_{\mu}^{-1})dg \\ = \sum_{x_{i}} \int_{K'x_{i}} c_{-\lambda}(x_{i}g_{\mu}^{-1})dg = \int_{H} c_{-\lambda}(hg_{\mu}^{-1})dh.
\end{multline*}
Finally, we have to multiply this by $meas.(K \cdot d_{\mu})$, see the previous proof for these values in each Case.

\textbf{Case 3:} Analogous to Case 1, except that $g_{-\mu} \cdot d_{0} = d_{\mu}$, so one uses $g_{-\mu}$ instead of $g_{\mu}$.

\textbf{Case 2 and Case 4:} We have $s_{0} = I_{2n} \neq d_{0}$.  In both cases, $s_{0} \in K \cdot d_{0}$, so $s_{0} = k_{0} \cdot d_{0}$ for some $k_{0} \in K$.  Then one can check that $\text{Stab}(d_{0}) = k_{0}^{-1}Hk_{0}$, so that $H_{\mu} = g_{\mu} k_{0}^{-1} H k_{0}$ in the proof above.  Using this, we have $KH_{\mu}^{-1} = Kk_{0}^{-1}Hk_{0}g_{\mu}^{-1} = KHk_{0}g_{\mu}^{-1}$, so repeating the arguments of Case 1, we have
\begin{equation*}
\int_{KHk_{0}g_{\mu}^{-1}} c_{\lambda}(g^{-1}) dg = \int_{KH} c_{-\lambda}(gk_{0}g_{\mu}^{-1})dg.
\end{equation*}
Exactly as argued in Case 1, we have
\begin{equation*}
\int_{KH} c_{-\lambda}(gk_{0}g_{\mu}^{-1})dg = \int_{H} c_{-\lambda}(hk_{0}g_{\mu}^{-1}) dh.
\end{equation*}
Finally, we have to multiply this by $meas.(K \cdot d_{\mu})$, see the previous proof for these values in each Case.

The last part follows since
\begin{equation*}
g_{0} = \text{diag}(1, \dots, 1) = \text{Id}_{2n},
\end{equation*}
and $c_{\lambda}(k_{0}h) = c_{\lambda}(h)$ as the characteristic functions are $K$-bi-invariant.

\end{proof}

\begin{proof}[Proof of Theorem \ref{genthm}]
Putting the previous two propositions together proves the theorem in the introduction.
\end{proof}

We will now focus on the special case $\mu = 0$, where we can obtain some more explicit results.  First, we will use measure-theoretic arguments to compute the RHS of the integral identity in Theorem \ref{genthm}.

\begin{lemma}\label{lemma}
We have the following:

\textbf{Case 1:} 
 \begin{equation*}
\int_{H} c_{\lambda}(h) dh =  \begin{cases} 0 & \text{if $\lambda \neq \mu^{2}$ for any $\mu$,}
\\
p^{2\langle\mu, \rho_{3}\rangle} \frac{v_{n}(p^{-2})}{v_{\mu}(p^{-2})} &\text{if $\lambda = \mu^{2}$ for some $\mu$.}
\end{cases}
\end{equation*}

\textbf{Case 2:}
\begin{equation*}
\int_{H} c_{\lambda}(h) dh = p^{2\langle \lambda, \rho_{2} \rangle} \frac{v_{2n}(p^{-1})}{v_{\lambda}(p^{-1})}.
\end{equation*}

\textbf{Case 3:}
\begin{equation*}
\int_{H} c_{\lambda}(h) dh = \begin{cases} 0 & \text{if $\lambda \neq \mu \bar{\mu}$ for any $\mu$,} \\
p^{2<\mu, \rho_{1}>} \frac{\phi_{n}(p^{-2})}{\phi_{n-l(\mu)}(p^{-2})(1-p^{-1})^{l(\mu)}v_{\mu^{+}}(p^{-1})} &\text{if $\lambda = \mu \bar{\mu}$ for some $\mu$.}
\end{cases}
\end{equation*}

\textbf{Case 4:}
\begin{equation*}
\int_{H} c_{\lambda}(h) dh = \sum_{\mu \cup \nu = \lambda} p^{\langle \mu, \rho_{3} \rangle + \langle \nu, \rho_{3} \rangle} \frac{v_{n}(p^{-1})^{2}}{v_{\mu}(p^{-1})v_{\nu}(p^{-1})}
\end{equation*}

\end{lemma}
\begin{proof}
\hfill

\textbf{Case 1:} 
Note first that the integral of the LHS is the measure of the intersection $H \cap Kp^{\lambda}K$.  We recall the Cartan decomposition of $G = Gl_{2n}(\mathbb{Q}_{p})$: 
\begin{equation*}
Gl_{2n}(\mathbb{Q}_{p})= \cup K p^{\lambda} K, \text{           (disjoint union)}
\end{equation*}
where $p^{\lambda}$ is the element $\text{diag}.(p^{\lambda_{1}}, \dots, p^{\lambda_{2n}})$ in $G$.  Similarly, we have the Cartan decomposition for $Gl_{n}(\mathbb{Q}_{p}(\sqrt{a}))$:
\begin{equation*}
Gl_{n}(\mathbb{Q}_{p}(\sqrt{a})) = \cup K' p^{\mu} K', \text{     (disjoint union)}
\end{equation*}
where $p^{\mu}$ is the element $\text{diag}.(p^{\mu_{1}}, \dots, p^{\mu_{n}})$ in $Gl_{n}(\mathbb{Q}_{p}(\sqrt{a}))$ and $K' = Gl_{n}(\mathbb{Z}_{p}(\sqrt{a}))$.  Note that under the isomorphism $Gl_{n}(\mathbb{Q}_{p}(\sqrt{a})) \rightarrow H$, $K'$ is mapped to  $K \cap H$, which is contained in $K$.  Also the element $\text{diag}.(p^{\mu_{1}}, \dots, p^{\mu_{n}}) \in Gl_{n}(\mathbb{Q}_{p}(\sqrt{a})) $ is mapped to the diagonal matrix $\text{diag}.(p^{\mu_{1}}, \dots, p^{\mu_{n}}, p^{\mu_{n}}, \dots, p^{\mu_{1}})$, which is an element of $Kp^{\lambda}K$, where $\lambda = \mu_{1} \mu_{1} \mu_{2} \mu_{2} \dots \mu_{n} \mu_{n}$.  Thus, $H$ may be realized inside $G$ as the disjoint union of the double cosets $\{(K \cap H) p^{(\mu_{1}, \dots, \mu_{n}, \mu_{n}, \dots, \mu_{1})} (K \cap H)\}$, where $\mu$ is a partition of length at most $n$.

This implies $H \cap K p^{\lambda} K$ is empty unless $\lambda = \mu^{2}$ for some partition $\mu$, which gives the vanishing part of the claim.  If $\lambda = \mu^{2}$, the integral is equal to $meas.((K \cap H) p^{(\mu_{1}, \dots, \mu_{n}, \mu_{n}, \dots, \mu_{1})} (K \cap H))$, which is equivalent to $meas.(K'p^{\mu}K')$ inside $Gl_{n}(\mathbb{Q}_{p}(\sqrt{a}))$.  We can compute this last quantity using \cite[Ch. V]{Mac}.  Applying that result to the group $Gl_{n}(\mathbb{Q}_{p}(\sqrt{a}))$, and noting that $p^{2}$ is the size of the residue field of $\mathbb{Q}_{p}(\sqrt{a})$ gives
\begin{equation*}
|K'p^{\mu}K'| = (p^{2})^{\langle\mu, \rho_{3}\rangle} \frac{\Big(\prod_{i=1}^{n} (1-p^{-2i})\Big)/ (1-p^{-2})^{n}}{\Big(\prod_{j \geq 0} \prod_{i=1}^{m_{j}(\mu)} (1-p^{-2i})\Big)/(1-p^{-2})^{n}} = p^{2\langle\mu, \rho_{3}\rangle} \frac{v_{n}(p^{-2})}{v_{\mu}(p^{-2})}.
\end{equation*}

\textbf{Case 2:}
Note that we have the following Cartan decompositions:
\begin{equation*}
G = Gl_{2n}(\mathbb{Q}_{p}(\sqrt{a})) = \bigcup_{\lambda \in \Lambda_{2n}} \Big(Gl_{2n}(\mathbb{Z}_{p}(\sqrt{a})) p^{\lambda} Gl_{2n}(\mathbb{Z}_{p}(\sqrt{a})) \Big)
\end{equation*}
and
\begin{equation*}
H = Gl_{2n}(\mathbb{Q}_{p}) = \bigcup_{\lambda \in \Lambda_{2n}} \Big(Gl_{2n}(\mathbb{Z}_{p}) p^{\lambda} Gl_{2n}(\mathbb{Z}_{p})\Big),
\end{equation*}
where in both cases the unions are disjoint.  Note also that $Gl_{2n}(\mathbb{Z}_{p}) = K' \subset K = Gl_{2n}(\mathbb{Z}_{p}(\sqrt{a}))$.  Thus, the intersection $Kp^{\lambda}K \cap H$ is exactly $K'p^{\lambda}K'$.  Finally, from \cite[Ch. V (2.9)]{Mac}, we have
\begin{equation*}
\text{measure of } K'p^{\lambda}K' = p^{2\langle \lambda, \rho_{2} \rangle} \frac{v_{2n}(p^{-1})}{v_{\lambda}(p^{-1})}=p^{2\langle \lambda, \rho_{2} \rangle} \frac{\phi_{2n}(p^{-1})}{\prod_{i \geq 0} \phi_{m_{i}(\lambda)}(p^{-1})},
\end{equation*}
as desired.

\textbf{Case 3:}
Note that we have the following Cartan decompositions:
\begin{equation*}
G= Gl_{2n}(\mathbb{Q}_{p}) = \bigcup_{\lambda \in \Lambda_{2n}} \Big(Gl_{2n}(\mathbb{Z}_{p}) p^{\lambda} Gl_{2n}(\mathbb{Z}_{p})\Big)
\end{equation*}
and
\begin{equation*}
H= Sp_{2n}(\mathbb{Q}_{p}) = \bigcup_{\substack{\lambda = \mu \bar{\mu}\\ \text{in } \Lambda_{2n}}} \Big( Sp_{2n}(\mathbb{Z}_{p}) p^{\lambda} Sp_{2n}(\mathbb{Z}_{p}) \Big)
\end{equation*}
where in both cases the unions are disjoint.  This implies that the intersection $Kp^{\lambda}K \cap H$ is zero if $\lambda \neq \mu \bar{\mu}$ for some $\mu$ giving the vanishing part of the result.  If $\lambda = \mu \bar{\mu}$, the intersection is $K'p^{\lambda}K'$.  We use \cite{MacP} (which deals with the general reductive $p$-adic group case) to compute
\begin{equation*}
\text{measure of } K'p^{\mu{\bar{\mu}}}K' = p^{2\langle \mu, \rho_{1}\rangle } \frac{\phi_{n}(p^{-2})}{\phi_{n-l(\mu)}(p^{-2})(1-p^{-1})^{l(\mu)}v_{\mu^{+}}(p^{-1})},
\end{equation*}
as desired.

\textbf{Case 4:}
Analogous to the arguments of the previous cases.  We will indicate the differences.  The double coset $Kp^{\lambda}K$ in $G = Gl_{2n}(\mathbb{Q}_{p})$ contains $Gl_{n}(\mathbb{Z}_{p})p^{\mu}Gl_{n}(\mathbb{Z}_{p}) \times Gl_{n}(\mathbb{Z}_{p})p^{\nu}Gl_{n}(\mathbb{Z}_{p})$, for any $\mu, \nu$ such that $\mu \cup \nu = \lambda$ (identifying this with its image in $G$ under the embedding of Case 4), since the element
\begin{equation*}
\left( \begin{array}{cc}
p^{\mu} & 0   \\
0 & p^{\nu}  \\  \end{array} \right)
\end{equation*}
is in $Kp^{\lambda}K$ and $(k_{1}, k_{2}) \in Gl_{2n}(\mathbb{Z}_{p})$ for any $k_{1}, k_{2} \in Gl_{n}(\mathbb{Z}_{p})$.  One can also easily show that $Kp^{\lambda}K$ only contains double cosets of $H$ of this form.  Finally, one computes the measures of $Gl_{n}(\mathbb{Z}_{p})p^{\mu}Gl_{n}(\mathbb{Z}_{p})$ and $Gl_{n}(\mathbb{Z}_{p})p^{\nu}Gl_{n}(\mathbb{Z}_{p})$ using \cite[Ch. V]{Mac}, for example, to obtain the result.

\end{proof}

We now discuss the related integral identities corresponding to the special case $\mu = 0$.

\begin{theorem}\label{symplectic}
Symplectic identity (Case 1), \cite[Corollary 14]{VV}.  Let $\lambda$ be a partition of length at most $2n$.  Then 
\begin{equation*}
\frac{1}{Z} \int_{T} P_{\lambda}^{(2n)}(x_{1}^{\pm 1}, \dots, x_{n}^{\pm 1} ;t) \tilde \Delta_{K}^{(n)}(x; t;\pm \sqrt{t},0,0) dT =  0, 
\end{equation*}
unless $\lambda = \mu^{2}$.  In this case, the nonzero value is 
\begin{equation*}
\frac{v_{n}(t^{2})}{v_{\mu}(t^{2})}
\end{equation*}
(here the normalization factor $Z = \int_{T} \tilde \Delta_{K}^{(n)}(x;t;\pm \sqrt{t},0,0) dT$).
\end{theorem}

\begin{theorem} \label{Kawanaka}
Finite-dimensional version of Kawanaka's identity (Case 2), \cite[Corollary 15]{VV}.  Let $\lambda$ be a partition of length at most $2n$.  Then
\begin{align*}
\frac{1}{Z} \int_{T} P_{\lambda}(x_{1}^{\pm 1}, \dots, x_{n}^{\pm 1};t) \tilde \Delta_{K}^{(n)}(x;t;1,\sqrt{t},0,0) &= \frac{v_{2n}(\sqrt{t})}{v_{\lambda}(\sqrt{t})} 
\end{align*}
(here the normalization factor $Z = \int_{T} \tilde \Delta_{K}^{(n)}(x;t;1,\sqrt{t},0,0) dT$).  
\end{theorem}

\begin{theorem} \label{other} 
(Case 3), \cite[Theorem 22]{VV}.  Let $\lambda$ be a weight of the double cover of $GL_{2n}$, i.e., a half-integer vector such that $\lambda_{i} - \lambda_{j} \in \mathbb{Z}$ for all $i,j$.  Then 
\begin{align*}
\frac{1}{Z}\int_{T} P_{\lambda}^{(2n)}( \cdots t^{\pm 1/2}z_{i} \cdots ;t) \tilde \Delta_{S}^{(n)}(x; t^{2}) dT &=0,
\end{align*}
unless $\lambda = \mu \bar{\mu}$.  In this case, the nonzero value is
\begin{align*}
\frac{\phi_{n}(t^{2})}{(1-t)^{n}v_{\mu}(t) (1+t)(1+t^{2}) \cdots (1+t^{n-l(\mu)})}
\end{align*}  
(here the normalization factor $Z = \int_{T} \tilde \Delta_{S}(z;t^{2})dT$).
\end{theorem}

\begin{theorem} \label{new}
(Case 4).  Let $\lambda$ be a partition of length at most $2n$.  Then 
\begin{equation*}
\frac{1}{Z} \int_{T} P_{\lambda}^{(2n)}(x_{1}^{\pm 1}, \dots, x_{n}^{\pm 1};t) \tilde \Delta_{K}^{(n)}(x;t; \sqrt{t}, \sqrt{t},0,0) dT = \sum_{\mu \cup \nu = \lambda} t^{-\langle \mu, \rho_{3} \rangle - \langle \nu, \rho_{3} \rangle + \langle \lambda, \rho_{2} \rangle}  \frac{v_{n}(t)^{2}}{v_{\mu}(t)v_{\nu}(t)}
\end{equation*}

\end{theorem}

Note that the second identity is \textit{not} a vanishing identity.  This identity is a finite-dimensional analog of a result of Kawanaka (see \cite{Ka1}, \cite{Ka2}).  Kawanaka's identity has an interesting representation-theoretic significance for general linear groups over finite fields: it encodes the fact that the symmetric space $Gl_{n}(\mathbb{F}_{p^{2}})/Gl_{n}(\mathbb{F}_{p})$ is multiplicity free.

We also note that the fourth identity is new (as far as we know); also if we set $t=0$, we recover Theorem \ref{symplecticgp}, part (1), so it is a generalization of that result.

We are now prepared to provide $p$-adic proofs of Theorems \ref{symplectic}, \ref{Kawanaka}, \ref{other}, and \ref{new}.
\begin{proof}[Proof of Theorem \ref{symplectic}]
The vanishing part of the statement is automatic from \textbf{Case 1} of Theorem \ref{genthm} and Lemma \ref{lemma}.  To obtain the nonzero value, let $\lambda = \mu^{2}$.  Then we can compute
\begin{multline*}
2\langle\mu, \rho_{3}\rangle = 2 \Big( (n-1)\mu_{1} + (n-3)\mu_{2} + \cdots + (1-n)\mu_{n} \Big) 
= (n-1)(\mu_{1} + \mu_{1}) + (n-3)(\mu_{2} + \mu_{2})+ \cdots + (1-n)(\mu_{n} + \mu_{n}) \\
= (n-1)(\lambda_{1} + \lambda_{2}) + (n-3)(\lambda_{3} + \lambda_{4}) + \cdots + (1-n)(\lambda_{2n-1} + \lambda_{2n}) \\
= \lambda_{1}(n-1/2) + \lambda_{2}(n- 3/2) + \cdots + \lambda_{2n}(1/2 - n) = \langle\lambda, \rho_{2}\rangle.
\end{multline*}
Thus, we obtain
\begin{equation*}
\frac{1}{Z} \int_{T} P_{\lambda}^{(2n)}(x_{i}^{\pm 1};p^{-1}) \tilde \Delta_{K}^{(n)}(x; \pm p^{-1/2},0,0;p^{-1}) dT = \begin{cases} 0 & \text{if $\lambda \neq \mu^{2}$ for any $\mu$,}
\\
 \frac{v_{n}(p^{-2})}{v_{\mu}(p^{-2})} &\text{if $\lambda = \mu^{2}$ for some $\mu$.}
\end{cases}
\end{equation*}
Thus the equation in the statement of Theorem \ref{symplectic} holds for all $t = p^{-1}$, for $p$ an odd prime, and the left hand side of the equation is a rational function in $t$.  This provides an infinite sequence of values for $t$ for which the equation holds, so in particular it holds for all values of $t$ as desired.
\end{proof}

\begin{proof}[Proof of Theorem \ref{Kawanaka}]
The identity follows from \textbf{Case 2} of Theorem \ref{genthm} and Lemma \ref{lemma}, as in the proof of Theorem \ref{symplectic} above.  Note that these arguments show that the theorem holds for all $t = p^{-2}$.  This provides an infinite sequence of values for $t$ for which the equation holds, so in particular it holds for all values of $t$ as desired.
\end{proof}

\begin{proof}[Proof of Theorem \ref{other}]
The identity follows from \textbf{Case 3} of Theorem \ref{genthm} and Lemma \ref{lemma}, as in the proof of Theorem \ref{symplectic} above.  If $\lambda = \mu\bar{\mu}$ for some $\mu$, the integral is non-vanishing.  The evaluation follows by noting that $2\langle \mu, \rho_{1} \rangle = \langle \lambda, \rho_{2} \rangle$.  Note that these arguments show that the theorem holds for all $t = p^{-1}$.  This provides an infinite sequence of values for $t$ for which the equation holds, so in particular it holds for all values of $t$ as desired.
\end{proof}

\begin{proof}[Proof of Theorem \ref{new}]
The identity follows from \textbf{Case 4} of Theorem \ref{genthm} and Lemma \ref{lemma}, as in the proof of Theorem \ref{symplectic} above.  Note that these arguments show that the theorem holds for all $t = p^{-1}$.  This provides an infinite sequence of values for $t$ for which the equation holds, so in particular it holds for all values of $t$ as desired.
\end{proof}

\begin{remarks}
In \textbf{Case 1}, the involution is $g \rightarrow g^{\star} = g^{-1}$ and the action is $g \cdot x = gxg^{*}$.  Then $H$ is the stabilizer in $G$ of $s_{0}$ under this action.  But $H = \{ g \in G | gs_{0}g^{*} = s_{0} \} = \{g \in G | g = s_{0}g^{*^{-1}}s_{0}^{-1} \}$.  So $H$ is the set of fixed points of the order $2$ homomorphism $g \rightarrow s_{0}g^{*^{-1}}s_{0}^{-1}$.  This provides an analog of Theorem (\ref{symplecticgp}), where one restricts $s_{\lambda}$ to the subgroup of fixed points of a suitable involution.  The other cases are analogous.
\end{remarks}

We now demonstrate how one can use Theorem \ref{new} to obtain a new Littlewood summation identity, namely that of Theorem \ref{Littlewood}.

\begin{proof}[Proof of Theorem \ref{Littlewood}]

From \cite[Lemma 7.18]{R}, we obtain the following integral identity
\begin{equation} \label{inf}
\frac{1}{Z} \int_{T} \prod_{j,k} \frac{1-tx_{j}y_{k}^{\pm 1}}{1-x_{j}y_{k}^{\pm 1}} \tilde \Delta_{K}(y;t;\sqrt{t}, \sqrt{t}, 0, 0) dT = \prod_{j<k} \frac{1-tx_{j}x_{k}}{1-x_{j}x_{k}} \prod_{j} \frac{1-tx_{j}^{2}}{(1-\sqrt{t}x_{j})(1-\sqrt{t}x_{j})},
\end{equation}
where the integral in the LHS is with respect to the $y$-variables.  Recall that the Cauchy identity for Hall-Littlewood functions is
\begin{equation*}
\sum_{\lambda} P_{\lambda}(x;t) Q_{\lambda}(y;t) = \prod_{i,j \geq 1} \frac{1-tx_{i}y_{j}}{1-x_{i}y_{j}},
\end{equation*}
where
\begin{equation*}
Q_{\lambda}(x;t) = b_{\lambda}(t)P_{\lambda}(x;t).
\end{equation*}
Using this on the product in the LHS of (\ref{inf}) allows us to rewrite that equation as 
\begin{multline} \label{inf2}
\sum_{\lambda} P_{\lambda}(x;t) \lim_{n \rightarrow \infty} \frac{1}{Z} \Bigg[ \int_{T} b_{\lambda}(t) P_{\lambda}(y_{1}^{\pm 1}, \dots, y_{n}^{\pm 1};t) \tilde \Delta_{K}^{(n)}(y;t;\sqrt{t}, \sqrt{t},0,0) dT \Bigg] \\
= \prod_{j<k} \frac{1-tx_{j}x_{k}}{1-x_{j}x_{k}} \prod_{j} \frac{1-tx_{j}^{2}}{(1-\sqrt{t}x_{j})^{2}} = \prod_{j<k} \frac{1-tx_{j}x_{k}}{1-x_{j}x_{k}} \prod_{j} \frac{1+\sqrt{t}x_{j}}{1-\sqrt{t}x_{j}}.
\end{multline}
Recall that by Theorem \ref{new}, we have
\begin{multline*}
\frac{1}{Z} \Bigg[ \int_{T} b_{\lambda}(t) P_{\lambda}(y_{1}^{\pm 1}, \dots, y_{n}^{\pm 1};t) \tilde \Delta_{K}^{(n)}(y;t;\sqrt{t}, \sqrt{t},0,0) dT \Bigg] \\= \sum_{\substack{\mu \cup \nu = \lambda \\ l(\mu), l(\nu) \leq n}} t^{-\langle \mu, \rho_{3} \rangle - \langle \nu, \rho_{3} \rangle + \langle \lambda, \rho_{2}\rangle} \frac{v_{n}(t)^{2} b_{\lambda}(t)}{v_{\mu}(t) v_{\nu}(t)}.
\end{multline*}

So we want to compute
\begin{equation} \label{limit}
 \lim_{n \rightarrow \infty} \sum_{\substack{\mu \cup \nu = \lambda \\ l(\mu), l(\nu) \leq n}} t^{-\langle \mu, \rho_{3} \rangle - \langle \nu, \rho_{3} \rangle + \langle \lambda, \rho_{2}\rangle} \frac{v_{n}(t)^{2} b_{\lambda}(t)}{v_{\mu}(t) v_{\nu}(t)}.
\end{equation}
First, note that
\begin{equation*}
\frac{v_{n}(t)^{2}b_{\lambda}(t)}{v_{\mu}(t)v_{\nu}(t)} = \frac{\phi_{n}(t)^{2} \prod_{i \geq 1} \phi_{m_{i}(\lambda)}(t)}{\prod_{i \geq 0} \phi_{m_{i}(\mu)}(t) \prod_{i \geq 0} \phi_{m_{i}(\nu)}(t)}.
\end{equation*}
Also
\begin{equation*}
\lim_{n \rightarrow \infty} \frac{\phi_{n}(t)^{2}}{\phi_{m_{0}(\mu)}(t) \phi_{m_{0}(\nu)}(t)} = \lim_{n \rightarrow \infty} (1-t^{m_{0}(\mu) + 1}) \cdots (1-t^{n})(1-t^{m_{0}(\nu) + 1}) \cdots (1-t^{n}) = 1,
\end{equation*}
since as $n \rightarrow \infty$, we have $m_{0}(\nu), m_{0}(\mu) \rightarrow \infty$ with $n-m_{0}(\mu), n - m_{0}(\nu)$ fixed.
Note that
\begin{equation*}
\rho_{2} = n^{2n}- \Big(\frac{1}{2}, \frac{3}{2}, \dots, \big(2n - \frac{1}{2}\big) \Big),
\end{equation*}
and
\begin{equation*}
\rho_{3} = n^{n} - (1,3,5, \dots, (2n-1)).
\end{equation*}
So for $\mu, \nu$ such that $\mu \cup \nu = \lambda$ we have
\begin{equation*}
-\langle \mu + \nu, n^{n} \rangle + \langle \lambda, n^{2n} \rangle = 0.
\end{equation*}

Thus, (\ref{limit}) is equal to
\begin{equation*}
 \sum_{\substack{\mu \cup \nu = \lambda }} t^{\langle \mu + \nu, (1,3,5, \dots) \rangle - \langle \lambda, (\frac{1}{2}, \frac{3}{2}, \dots) \rangle} \frac{b_{\lambda}(t)}{b_{\mu}(t) b_{\nu}(t)}.
\end{equation*}
Putting this back into (\ref{inf2}), the resulting Littlewood identity is
\begin{equation*}
\sum_{\mu, \nu} P_{\mu \cup \nu}(x;t) t^{\langle \mu + \nu, \rho \rangle - \frac{1}{2}\langle \mu \cup \nu, \rho \rangle} \frac{b_{\mu \cup \nu}(t)}{b_{\mu}(t) b_{\nu}(t)} = \prod_{j<k} \frac{1-tx_{j}x_{k}}{1-x_{j}x_{k}} \prod_{j} \frac{1+\sqrt{t}x_{j}}{1-\sqrt{t}x_{j}}.
\end{equation*}

\end{proof}

\section{Generalized integral identity}

In this section, we deal only with the symplectic case, \textbf{Case 1}; the notation is as in that case.  We will prove some stronger results by extending the methods above.  

\noindent Throughout this section, we fix $l(\lambda) \leq 2n$ and $l(\mu) \leq n$.  Then, by Theorem \ref{genthm} (Case 1), we have
\begin{multline*}
\frac{1}{Z} \int_{T} P_{\lambda}^{(2n)}(x_{i}^{\pm 1}; p^{-1}) K_{\mu}^{BC_{n}}(x; p^{-1}; \pm p^{-1/2},0,0) \tilde \Delta_{K}^{(n)}(x; \pm p^{-1/2},0,0; p^{-1}) dT \\
= p^{\langle\mu, \rho_{1}\rangle - \langle\lambda, \rho_{2}\rangle} \frac{V_{0}}{V_{\mu}} \int_{H} c_{\lambda}(hg_{\mu}^{-1}) dh.
\end{multline*}
Using the Cartan decomposition for $(Gl_{n}(\mathbb{Q}_{p}(\sqrt{a})), Gl_{n}(\mathbb{Z}_{p}(\sqrt{a})))$ and the embedding into $Gl_{2n}(\mathbb{Q}_{p})$, we have
\begin{equation*}
\int_{H} c_{\lambda}(hg_{\mu}^{-1}) dh = \sum_{\substack{\beta \in \Lambda_{2n} \\ \beta = \nu_{1} \dots \nu_{n} \nu_{n} \dots \nu_{1} \\ \text{for some $\nu$}}} \int_{K' p^{\beta} K'} c_{\lambda}(hg_{\mu}^{-1}) dh;
\end{equation*}
also note that
\begin{equation*}
\int_{K'p^{\beta}K'} c_{\lambda}(hg_{\mu}^{-1}) dh = meas.(Kp^{\lambda}Kg_{\mu} \cap K'p^{\beta}K'),
\end{equation*}
where the measure is with respect to the measure on $H$.  Thus,
\begin{multline} \label{negpts}
\frac{1}{Z} \int_{T} P_{\lambda}^{(2n)}(x_{i}^{\pm 1}; p^{-1}) K_{\mu}^{BC_{n}}(x; p^{-1}; \pm p^{-1/2},0,0) \tilde \Delta_{K}^{(n)}(x; \pm p^{-1/2},0,0; p^{-1}) dT \\
= p^{\langle\mu, \rho_{1}\rangle - \langle\lambda, \rho_{2}\rangle} \frac{V_{0}}{V_{\mu}} \sum_{\substack{\beta \in \Lambda_{2n} \\ \beta = \nu_{1} \dots \nu_{n} \nu_{n} \dots \nu_{1} \\ \text{for some $\nu$}}} meas.(Kp^{\lambda}Kg_{\mu} \cap K'p^{\beta}K').
\end{multline}

\begin{lemma}
Let $\beta = \nu_{1} \dots \nu_{n} \nu_{n} \dots \nu_{1} \in \Lambda_{2n}$ have at least one negative part.  Then $meas.(Kp^{\lambda}Kg_{\mu} \cap K'p^{\beta}K') = 0$.
\end{lemma}
\begin{proof}
Note that if $Kp^{\lambda}K \cap K'p^{\beta}K'g_{\mu}^{-1} \neq \emptyset $, then $Kp^{\lambda}K \cap p^{\beta}K' g_{\mu}^{-1} \neq \emptyset$.  We will show that $Kp^{\lambda}K \cap p^{\beta}K' g_{\mu}^{-1} = \emptyset$, which proves the claim.  

Note first that $g_{\mu}^{-1} = \text{diag}.(1, \dots, 1,p^{\mu_{n}}, \dots, p^{\mu_{1}})$.  We will write $\bar{\mu} = (\mu_{n}, \dots, \mu_{1})$.  Suppose for contradiction that
\begin{equation*}
 k'= \left( \begin{array}{cc}
i & j   \\
aw_{n}jw_{n} & w_{n}iw_{n}  \\  \end{array} \right)
\end{equation*}
is an element in $K'$ such that $p^{\beta}k'g_{\mu}^{-1} \in Kp^{\lambda}K$.  By a direct computation we have
\begin{equation*}
p^{\beta}k'g_{\mu}^{-1} = \left( \begin{array}{cc}
p^{\nu} & 0   \\
0 & p^{\bar{\nu}} \\  \end{array} \right) 
\left( \begin{array}{cc}
i & j   \\
aw_{n}jw_{n} & w_{n}iw_{n}  \\  \end{array} \right)
\left( \begin{array}{cc}
1 & 0   \\
0 & p^{\bar{\mu}} \\  \end{array} \right) 
= 
\left( \begin{array}{cc}
p^{\nu}i & p^{\nu}jp^{\bar{\mu}}   \\
p^{\bar{\nu}}aw_{n}jw_{n} & p^{\bar{\nu}}w_{n}iw_{n}p^{\bar{\mu}} \\  \end{array} \right).
\end{equation*}
Now noting that $p^{\bar{\nu}}w_{n} = w_{n}p^{\nu}$, the above becomes
\begin{equation*}
\left( \begin{array}{cc}
p^{\nu}i & p^{\nu}jp^{\bar{\mu}}   \\
aw_{n}p^{\nu}jw_{n} & w_{n}p^{\nu}ip^{\mu}w_{n} \\  \end{array} \right).
\end{equation*}
Since $p^{\beta}k'g_{\mu}^{-1} \in Kp^{\lambda}K \subset M_{2n}(\mathbb{Z}_{p})$, it follows that $p^{\nu}i$ and $p^{\nu}j$ are in $M_{n}(\mathbb{Z}_{p})$.  Since $\nu_{n} < 0$, it follows that the $n$-th row of $k'$ has entries all of which are divisible by $p$ in $\mathbb{Z}_{p}$.  
Let $B$ be the matrix obtained from $k'$ by dividing the $n$-th row by $p$; note that $B \in M_{2n}(\mathbb{Z}_{p})$.  Then 
\begin{equation*}
\det(k') = p \det(B) \in p \cdot \mathbb{Z}_{p},
\end{equation*}
which contradicts $|\det(k')| = 1$.  
\end{proof}

\noindent Thus, using the previous lemma, (\ref{negpts}) now becomes 
\begin{multline} \label{pospts}
\frac{1}{Z} \int_{T} P_{\lambda}^{(2n)}(x_{i}^{\pm 1}; p^{-1}) K_{\mu}^{BC_{n}}(x; p^{-1}; \pm p^{-1/2},0,0) \tilde \Delta_{K}^{(n)}(x; \pm p^{-1/2},0,0; p^{-1}) dT \\
= p^{\langle\mu, \rho_{1}\rangle - \langle\lambda, \rho_{2}\rangle} \frac{V_{0}}{V_{\mu}} \sum_{\substack{\beta \in \Lambda_{2n} ^{+}\\ \beta = \nu_{1} \dots \nu_{n} \nu_{n} \dots \nu_{1} \\ \text{for some $\nu$}}} meas.(Kp^{\lambda}Kg_{\mu} \cap K'p^{\beta}K').
\end{multline}

\noindent \textbf{Littlewood-Richardson coefficients and Hall polynomials:}

Recall that the Littlewood-Richardson coefficient $c^{\lambda}_{\mu \nu}$ is equal to the number of tableaux $T$ of shape $\lambda - \mu$ and weight $\nu$ such that $w(T)$, the word of $T$, is a lattice permutation.  We have
\begin{equation*}
s_{\mu}s_{\nu} = \sum_{\lambda} c^{\lambda}_{\mu \nu} s_{\lambda},
\end{equation*}
where $s_{\mu}$ is the Schur function (see \cite{Mac} for more details).

We briefly recall the Hall polynomials $g_{\mu \nu}^{\lambda}(q)$ \cite[Chs. II and V]{Mac}.  Let $\mathcal{O}$ be a complete (commutative) discrete valuation ring, $\mathcal{P}$ its maximal ideal and $k = \mathcal{O}/\mathcal{P}$ the residue field.  We assume $k$ is a finite field.  Let $q$ be the number of elements in $k$.  Let $M$ be a finite $\mathcal{O}$-module of type $\lambda$.  Then the number of submodules of $N$ of $M$ with type $\nu$ and cotype $\mu$ is a polynomial in $q$, called the Hall polynomial, denoted $g_{\mu \nu}^{\lambda}(q)$.  One can consider our motivating case of $\mathbb{Q}_{p}$ and its ring of integers $\mathcal{O} = \mathbb{Z}_{p}$ and $G = Gl_{n}(\mathbb{Q}_{p})$, so that $q=p$.  Then they are also the structure constants for the ring $\mathcal{H}(G^{+},K)$.  That is, for $\mu, \nu \in \Lambda_{2n}^{+}$, we have
\begin{equation*}
c_{\mu} \star c_{\nu} = \sum_{\lambda \in \Lambda_{2n}^{+}} g_{\mu \nu}^{\lambda}(p) c_{\lambda}.
\end{equation*}
Note that, in particular,
\begin{equation*}
g_{\mu \nu}^{\lambda}(p) = (c_{\mu} \star c_{\nu})(p^{\lambda}) = \int_{G} c_{\mu}(p^{\lambda}y^{-1})c_{\nu}(y)dy = meas.(p^{\lambda}Kp^{-\nu}K \cap Kp^{\mu}K).
\end{equation*}

Several important facts are known (see \cite[Ch. II]{Mac}):
\begin{enumerate}
\item If $c^{\lambda}_{\mu \nu} = 0$, then $g^{\lambda}_{\mu \nu}(t) = 0$ as a function of $t$.
\item If $c^{\lambda}_{\mu \nu} \neq 0$, then $g^{\lambda}_{\mu \nu}(t)$ has degree $n(\lambda) - n(\mu) - n(\nu)$ and leading coefficient $c^{\lambda}_{\mu \nu}$, where the notation $n(\lambda) = \sum (i-1) \lambda_{i}$.
\item We have $g^{\lambda}_{\mu \nu}(t) = g^{\lambda}_{\nu \mu}(t)$.
\end{enumerate}

Also if one multiplies two Hall-Littlewood polynomials, and expands the result in the Hall-Littlewood basis, one has
\begin{equation*}
P_{\mu}(x;t) P_{\nu}(x;t) = \sum_{\lambda} f^{\lambda}_{\mu \nu}(t) P_{\lambda}(x;t),
\end{equation*}
with $f^{\lambda}_{\mu \nu}(t) = t^{n(\lambda) - n(\mu) - n(\nu)} g^{\lambda}_{\mu \nu}(t^{-1})$.

\begin{lemma}\label{2d}
Let $\lambda, \mu, \beta \in \Lambda_{2n}$.  Then we have
\begin{equation*}
\int_{G} c_{-\mu}(g') \int_{G} c_{\beta}(g) c_{\lambda}(gg') dg dg' = meas.(Kp^{-\mu}K) \int_{G} c_{-\lambda}(p^{\mu}g^{-1})c_{\beta}(g)dg
\end{equation*}
\end{lemma}
\begin{proof}
Write $Kp^{-\mu}K$ as the disjoint union $\cup k_{i}p^{-\mu}K$, where $k_{i} \in K$.  Then 
\begin{multline*}
\int_{G} c_{-\mu}(g') \int_{G} c_{\beta}(g) c_{\lambda}(gg') dg dg' = \int_{Kp^{-\mu}K}  \int_{G} c_{\beta}(g) c_{\lambda}(gg')dg dg' = \sum_{k_{i}p^{-\mu}} \int_{K} \int_{G} c_{\beta}(g)c_{\lambda}(gk_{i}p^{-\mu}k) dg dk \\
= \sum_{k_{i}p^{-\mu}} \int_{G} c_{\beta}(g) c_{\lambda}(gk_{i}p^{-\mu}) dg = \sum_{k_{i}p^{-\mu}}\int_{G} c_{\beta}(yk_{i}^{-1})c_{\lambda}(yp^{-\mu})dy = \sum_{k_{i}p^{-\mu}} \int_{G} c_{\beta}(y)c_{\lambda}(yp^{-\mu})dy \\ 
= meas.(Kp^{-\mu}K) \int_{G} c_{\beta}(g)c_{\lambda}(gp^{-\mu})dg = meas.(Kp^{-\mu}K) \int_{G} c_{-\lambda}(p^{\mu}g^{-1})c_{\beta}(g)dg.
\end{multline*}
\end{proof}

\begin{proposition}
Let $\lambda \in \Lambda_{2n}^{+}$ and $\mu \in \Lambda_{n}^{+}$ and fix a prime $p \neq 2$.  Suppose $g^{\lambda}_{\mu, \beta}(p) = 0$ for all $\beta \in \Lambda_{2n}^{+}$ with all parts occurring with even multiplicity.  Then the integral 
\begin{equation*}
\frac{1}{\int_{T} \tilde \Delta_{K}^{(n)}(x; \pm p^{-1/2},0,0; p^{-1}) dT} \int_{T} P_{\lambda}^{(2n)}(x_{i}^{\pm 1}; p^{-1}) K_{\mu}^{BC_{n}}(x; p^{-1}; \pm p^{-1/2},0,0) \tilde \Delta_{K}^{(n)}(x; \pm p^{-1/2},0,0; p^{-1}) dT 
\end{equation*}
vanishes.  
\end{proposition}
\begin{proof}
The starting point is (\ref{pospts}) from the discussion above, recall that we have
\begin{multline*}
\frac{1}{Z} \int_{T} P_{\lambda}^{(2n)}(x_{i}^{\pm 1}; p^{-1}) K_{\mu}^{BC_{n}}(x; p^{-1}; \pm p^{-1/2},0,0) \tilde \Delta_{K}^{(n)}(x; \pm p^{-1/2},0,0; p^{-1}) dT \\
= p^{\langle\mu, \rho_{1}\rangle - \langle\lambda, \rho_{2}\rangle} \frac{V_{0}}{V_{\mu}} \sum_{\substack{\beta \in \Lambda_{2n} ^{+}\\ \beta = \nu_{1} \dots \nu_{n} \nu_{n} \dots \nu_{1} \\ \text{for some $\nu$}}} meas.(Kp^{\lambda}Kg_{\mu} \cap K'p^{\beta}K').
\end{multline*}
Now if we write
\begin{equation*}
(Kp^{\lambda}Kg_{\mu} \cap K'p^{\beta}K') = \cup K' x_{i},
\end{equation*}
a disjoint union and $x_{i} \in p^{\beta}K'$, then the above measure is the number of $x_{i}$'s.  But we also have
\begin{equation*}
\cup K x_{i} \subset (Kp^{\lambda}Kg_{\mu} \cap Kp^{\beta}K),
\end{equation*}
and the union is disjoint ($k_{1}x_{i} = k_{2}x_{j}$ implies $k_{2}^{-1}k_{1}x_{i} = x_{j}$, but $x_{i},x_{j} \in H$ so $k_{2}^{-1}k_{1} \in K'$, a contradiction to the definition of the $x_{j}$'s).  Thus,
\begin{equation*}
meas.(Kp^{\lambda}Kg_{\mu} \cap K'p^{\beta}K') = \# \{x_{i} \} = meas.(\cup Kx_{i}) \leq meas.(Kp^{\lambda}Kg_{\mu} \cap Kp^{\beta}K),
\end{equation*}
so that
\begin{equation*}
\int_{K'p^{\beta}K'}c_{\lambda}(hg_{\mu}^{-1}) dh \leq \int_{Kp^{\beta}K}c_{\lambda}(gg_{\mu}^{-1})dg = \int_{G} c_{\beta}(g) c_{\lambda}(gg_{\mu}^{-1})dg = \int_{G} c_{-\lambda}(g_{\mu}g^{-1})c_{\beta}(g)dg.
\end{equation*}
Recall that $g_{\mu} = p^{(0^{n},-\mu_{n}, \dots, -\mu_{1})}$.  By Lemma (\ref{2d}), we have
\begin{equation*}
\int_{G} c_{-\lambda}(g_{\mu}g^{-1})c_{\beta}(g)dg = \frac{1}{meas.(Kp^{\mu 0^{n}}K) } \int_{G} c_{\mu 0^{n}}(g') \int_{G} c_{\beta}(g) c_{\lambda}(gg')dgdg'.
\end{equation*}
But, using a change of variables, we have
\begin{multline*}
 \int_{G} c_{\mu 0^{n}}(g') \int_{G} c_{\beta}(g) c_{\lambda}(gg')dgdg' = \int_{G} c_{\mu0^{n}}(g') \int_{G} c_{\beta}(yg'^{-1})c_{\lambda}(y) dy dg' \\
 = \int_{G} c_{\mu0^{n}}(g') \int_{G} c_{\beta}(y^{-1}g'^{-1})c_{-\lambda}(y) dy dg'  
 = \int_{G} c_{-\lambda}(y)\int_{G} c_{\mu 0^{n}}(g')  c_{-\beta}(g'y) dg'dy \\
 = meas.(Kp^{-\lambda}K) \int_{G} c_{\beta}(p^{\lambda}g^{-1})c_{\mu0^{n}}(g) dg = meas.(Kp^{-\lambda}K) g_{\beta, \mu0^{n}}^{\lambda}(p).
 \end{multline*}
Thus,
\begin{equation*}
\int_{G} c_{-\lambda}(g_{\mu}g^{-1})c_{\beta}(g)dg = \frac{meas.(Kp^{-\lambda}K)}{meas.(Kp^{\mu0^{n}}K)} g_{\beta, \mu0^{n}}^{\lambda}(p) = \frac{meas.(Kp^{-\lambda}K)}{meas.(Kp^{\mu0^{n}}K)} g_{ \mu0^{n}, \beta}^{\lambda}(p),
\end{equation*}
where the last equality follows from Fact 3 about Hall polynomials above.  

Thus, we have
\begin{equation*}
\int_{H} c_{\lambda}(hg_{\mu}^{-1})dh = \sum_{\substack{\beta \in \Lambda_{2n} ^{+}\\ \beta = \nu_{1} \dots \nu_{n} \nu_{n} \dots \nu_{1} \\ \text{for some $\nu$}}} meas.(Kp^{\lambda}Kg_{\mu} \cap K'p^{\beta}K') \leq \sum_{\substack{\beta \in \Lambda_{2n}^{+} \\ \beta = \nu^{2} \\ \text{for some $\nu$}}} \frac{meas.(Kp^{-\lambda}K)}{meas.(Kp^{\mu0^{n}}K)} g_{ \mu0^{n}, \beta}^{\lambda}(p).
\end{equation*}

Since by assumption $g^{\lambda}_{\mu, \beta}(p) = 0$ for all $\beta= \nu^{2} \in \Lambda_{2n}^{+}$, the result follows.
\end{proof}

\begin{proposition} \label{HLSchurvan}
Let $\lambda, \mu \in \Lambda_{n}^{+}$.  Then the integral
\begin{equation*}
\frac{1}{\int_{T} \tilde \Delta_{K}^{(n)}(x;0,0,0,0;t) dT}\int_{T} s_{\lambda}^{(2n)}(x_{i}^{\pm 1}) sp_{\mu}(x_{1}, \dots, x_{n}) \tilde \Delta_{K}^{(n)}(x;0,0,0,0;t) dT
\end{equation*}
vanishes if and only if the integral
\begin{equation*}
\frac{1}{\int_{T} \tilde \Delta_{K}^{(n)}(x; \pm \sqrt{t},0,0;t) dT} \int_{T} P_{\lambda}^{(2n)}(x_{i}^{\pm 1};t) K_{\mu}^{BC_{n}}(x;t; \pm \sqrt{t},0,0) \tilde \Delta_{K}^{(n)}(x; \pm \sqrt{t},0,0;t)dT
\end{equation*}
vanishes as a rational function of $t$.
\end{proposition}
\begin{proof}
The ``if" direction follows by setting $t=0$ in the Hall-Littlewood polynomial integral to obtain the Schur case.  We consider the other direction: i.e., suppose the integral involving Schur polynomials vanishes.  We will show the integral involving the Hall-polynomial vanishes.

Fix an odd prime $p$.  By Theorem \ref{WL}, since the above Schur integral vanishes, we must have $c^{\lambda}_{\mu, \beta} = 0$ for all $\beta \in \Lambda_{2n}^{+}$ with all parts occurring with even multiplicity.  

By Fact 1 about Hall polynomials above, this implies $g^{\lambda}_{\mu, \beta}(p)=0$ for all $\beta \in \Lambda_{2n}^{+}$ with all parts occurring with even multiplicity.  Thus, by the previous proposition, we have
\begin{equation*}
\frac{1}{\int_{T} \tilde \Delta_{K}^{(n)}(x; \pm p^{-1/2},0,0; p^{-1}) dT} \int_{T} P_{\lambda}^{(2n)}(x_{i}^{\pm 1}; p^{-1}) K_{\mu}^{BC_{n}}(x; p^{-1}; \pm p^{-1/2},0,0) \tilde \Delta_{K}^{(n)}(x; \pm p^{-1/2},0,0; p^{-1}) dT = 0.
\end{equation*}
This shows that the integral in question vanishes for all values $t = p^{-1}$, $p$ an odd prime.  Thus it vanishes for all values of $t$.  
\end{proof}

We are now ready to provide a proof of Theorem \ref{genvan}, mentioned in the Introduction.

\begin{proof}[Proof of Theorem \ref{genvan}]
Follows from Proposition \ref{HLSchurvan} and Theorem \ref{WL}.
\end{proof}

\begin{example}[1]
Let $\lambda$ have all parts occurring with even multiplicity, and $\mu = (r)$ only one part (assume $r \neq 0)$.  Let $\beta$ have all parts occurring with even multiplicity.  We have $g^{\lambda}_{\beta, (r)}(t) = 0$ unless $\lambda - \beta$ is a horizontal $r$-strip \cite{Mac}.  But $\lambda - \beta$ is a horizontal-strip if and only if $\lambda_{1} \geq \beta_{1} \geq \lambda_{2} \geq \beta_{2} \cdots$ (interlaced), so $\lambda = \beta$.  Thus $g^{\lambda}_{(r), \beta}(t) = 0$ for all $\beta$ with all parts occurring with even multiplicity.  So for these conditions on $\lambda$, $\mu$, the integral of Theorem \ref{genvan} part (1) vanishes.  
\end{example}

\begin{example}[2]
Let $\mu = 0$.  Then by \cite{Mac}, $c^{\lambda}_{\mu, \beta} = 0$ for all $\beta \neq \lambda$.  Thus, the integral of Theorem \ref{genvan} part (1) vanishes unless $\lambda = \beta$, where $\beta$ has all parts occuring with even multiplicity.
\end{example}

\end{document}